\UseRawInputEncoding
\pdfoutput=1

\documentclass[10pt]{amsart}

\usepackage{graphicx}

\usepackage{amsmath}

\usepackage{amssymb}

\newtheorem{thm}{Theorem}

\newtheorem{thmx}{Theorem} %non contato
\newtheorem{prop}[thm]{Proposition}
\newtheorem{lemma}[thm]{Lemma}
\newtheorem{cor}[thm]{Corollary}

\theoremstyle{definition}
\newtheorem{defn}{Definition}
\newtheorem{ex}[thm]{Example}

\theoremstyle{remark}
\newtheorem{remark}{Remark}

\def\T{{}^t\!}
\def\C{\mathbb{C}}
\def\R{\mathbb{R}}

\def\CP{\mathbb{CP}}
\def\Z{\mathbb{Z}}

\def\Q{\mathbb{Q}}

\def\Sp{\mathrm{Sp}}

\def\SL{\mathrm{SL}}

\def\det{\mathrm{det\,}}
\def\dim{\mathrm{dim\,}}

\def\TT{\mathcal {T\,}}
\def\TTT{\widetilde{\mathcal {T\,}}}

\begin{document}
\title[]{Conformal geometry of isotropic curves\\ in the complex quadric}
%\title[]{Complex conformal geometry of isotropic curves in the nonsingular quadric}
%\title[]{Conformal geometry of isotropic curves\\ in the nonsingular complex quadric}
%\title[]{On the holomorphic conformal geometry of \\ isotropic curves in the complex quadric}
%\title[]{Holomorphic conformal geometry of \\ null curves in the complex quadric}
%\title[]{The conformal geometry of isotropic curves \\ in the complex quadric}
%\title[]{Holomorphic conformal geometry of \\ isotropic curves in the complex quadric}
%\title[]{Conformal geometry of holomorphic \\ null curves in the complex quadric}
%\title[]{Conformal geometry of null holomorphic curves in the complex quadric}
%\title[]{Holomorphic conformal geometry of isotropic curves in the complex 3-hyperquadric}
%\title[]{On the holomorphic conformal geometry of isotropic curves in the complex hyperquadric $\Q_3$}
%\title[]{Holomorphic conformal geometry of holomorphic null curves in the complex quadric}
%\title[]{On the conformal geometry of holomorphic \\ null curves in the complex quadric}

\author{Emilio Musso}
\address{(E. Musso) Dipartimento di Scienze Matematiche, Politecnico di Torino,
Corso Duca degli Abruz\-zi 24, I-10129 Torino, Italy}
\email{emilio.musso@polito.it}

\author{Lorenzo Nicolodi}
\address{(L. Nicolodi) Di\-par\-ti\-men\-to di Scienze Ma\-te\-ma\-ti\-che, Fisiche e Informatiche,
Uni\-ver\-si\-t\`a di Parma, Parco Area delle Scienze 53/A,
I-43124 Parma, Italy}
\email{lorenzo.nicolodi@unipr.it}

\thanks{Authors partially supported by
PRIN 2017 ``Real and Complex Manifolds: Topology, Geometry and Holomorphic Dynamics''
(protocollo 2017JZ2SW5-004);
by the GNSAGA of INDAM; and by the FFABR Grant 2017 of MIUR. The present research was partially
supported by MIUR grant ``Dipartimenti di Eccellenza'' 2018–2022, CUP: E11G18000350001, DISMA, Politecnico
di Torino}


%\date{Version of \today}
%\date{Version of September 28, 2021}

%\dedicatory{}


\begin{abstract}
Let $\Q_3$ be the complex 3-quadric endowed with its standard complex conformal structure.
We study the complex conformal geometry of isotropic curves in $\Q_3$. 
By an isotropic curve we mean a nonconstant holomorphic map from a Riemann surface into
$\Q_3$, null with respect to the conformal structure of $\Q_3$.
The relations between isotropic curves and a number of 
relevant classes of surfaces in Riemannian and Lorentzian spaceforms are discussed.
\end{abstract}

\maketitle

\section{Introduction}\label{s:intro}
Let $\Q_3$ be the 
%3-dimensional 
nonsingular complex hyperquadric of $\CP^4$, regarded as the Grassmannian of Lagrangian 2-planes of $\C^4$
and equipped with its canonical $\Sp(2,\C)$-invariant complex conformal structure ${\bf g}$ (cf. \cite{JM,KO1982,LeBrun82}). 
An \textit{isotropic curve} is a nonconstant  holomorphic map $f : S \to \Q_3$ from a Riemann surface $S$ into $\Q_3$,
such that $f^*({\bf g})=0$. 
(The complex null geodesics, i.e., the projective lines in $\CP^4$ contained in $\Q_3$, are
excluded from our consideration.) 
The purpose of this paper is to provide a systematic study of the complex conformal geometry of isotropic curves. 
%The choice of the topic 
This topic was inspired 
%and influenced 
by Chern's article \cite{Ch-S4} and Bryant's papers
\cite{Br-S4,Br1988,Br2018}.

\vskip0.1cm
The main motivation comes from differential geometry,
and more specifically from surface geometry. 
In fact, the point of view of isotropic curves provides a unifying framework for 
a number of different
classes of immersed surfaces
according to the following scheme:
an isotropic curve is naturally associated to any immersion of the class; and conversely an immersion 
can be recovered from the isotropic curve. 
Examples that fit into this scheme include minimal surfaces in Euclidean space $\R^3$, maximal surfaces in 
Minkowski space $\R^{1,2}$,
surfaces of constant mean curvature one (CMC 1) and flat fronts in hyperbolic space  
${\mathcal H}^3$, spacelike CMC 1 surfaces and flat fronts in de~Sitter space ${\mathcal H}^{1,2}$, and 
superminimal immersions in the four-sphere $S^4$. 
For minimal surfaces, this amounts essentially to the classical description of a minimal surface $f: S \to \mathbb R^3$
as the real part of a holomorphic null curve $F : S \to \C^3$
%which goes back to the work of Weierstrass and Goursat 
(cf. \cite{DHS,Go,W}). 
For the case of maximal surfaces in $\R^{1,2}$, see \cite{KO1983,Lee2005, Palmer1990}; for the case of 
CMC 1 surfaces in hyperbolic space, besides Bryant's seminal paper \cite{Br1987}, see 
\cite{Bo-Pe2005,RUY,Small1994,ST,UY-Annals};
for flat fronts in hyperbolic 3-space, see \cite{ET,GMM-MA,KUY-PJM,KUY-O,MN-Lincei,ST};
for spacelike CMC 1 surfaces and flat fronts in de~Sitter space, see \cite{Lee2005,Palmer1990}.
Finally, for the case of superminimal surfaces in $S^4$, see \cite{Br-S4,Ch-S4}. 
Other references will be given throughout the paper.

\vskip0.1cm

\noindent \textit{Description of results}. Section \ref{s:pre} collects some background 
material about the holomorphic conformal structure of $\Q_3$ and the contact structure of $\CP^3$. 
In particular, we  briefly discuss the projections of suitable Zariski open sets of $\Q_3$ onto $\mathbb R^3$, 
$\mathbb R^{1,2}$, ${\mathcal H}^{3}$
and ${\mathcal H}^{1,2}$, the projections of suitable Zariski open sets of $\CP^3$ onto ${\mathcal H}^{3}$
and ${\mathcal H}^{1,2}$, and reformulate the twistor projections of $\CP^3$ onto $S^4$.

Section \ref{s:proj-struct} 
is concerned with
%discusses
complex projective structures on a Riemann surface (cf. \cite{Ca5,K1,LP2009,OvTa2005}).
We provide the needed background on projective structures and
show that 
naturally associated to a 
projective structure on a Riemann surface
there is an
operator ${\mathfrak d}$ from quartic 
%meromorphic differentials 
to quadratic meromorphic differentials.

Section \ref{s: isotropic curves} studies the complex conformal geometry of isotropic curves by the method of moving frames
and discusses the relation of isotropic curves with surface geometry.
First, we briefly recall the construction of the branched conformal (spacelike) immersions and fronts in 
$\R^3$, $\R^{2,1}$, ${\mathcal H}^3$, ${\mathcal H}^{2,1}$, and $S^4$ associated to an isotropic curve $f$, 
the {\it 
%associated 
surfaces tamed} by $f$. 
The main properties of these surfaces are collected in Theorem \ref{thm:associatedsurfaces}.
Next, generalizing the classical Goursat transformation for minimal surfaces (cf. \cite{Go,JMN}), we define the conformal 
Goursat transformation 
for surfaces tamed by isotropic curves. 
The close relationship between isotropic curves and projective structures on a Riemann surface
is established in Theorem \ref{thm:FourthReduction},
which shows that to any isotropic curve $f: S\to \Q_3$ one can associate a meromorphic projective 
structure and a meromorphic quartic differential $\delta$ on $S$. 
We then characterize the conformal cycles, i.e., the isotropic curves all of whose points are 
zeros of the quartic differential (heptactic points),
and describe their corresponding tamed surfaces. 
%
%The zeros of the quartic differential are called heptactic points.
%
Next, we introduce the notion of
osculating cycle for an isotropic curve $f$ and geometrically characterize 
%the zeros of the quartic differential 
%points at which the quartic differential vanishes identically 
heptactic points
in terms of the order of contact of the isotropic curve with the osculating cycle. 
In Theorem \ref{thm:equivalence}, we solve the equivalence problem for generic isotropic curves 
in terms of 
%by means of
the meromorphic differentials 
$\delta$ and ${\mathfrak d}(\delta)$. 
The meromorphic function ${\mathfrak d}^2(\delta)^2/\delta$ is called the bending of the isotropic curve.
In Theorem \ref{thm:def-rig}, 
viewing $\Q_3$ as a homogeneous space of $\Sp(2,\C)$,
we address the question of rigidity and deformation for isotropic curves
%consider rigidity and deformation 
according to the general deformation theory of submanifolds in homogeneous spaces
as formulated by Cartan \cite{Ca6}  and further developed by Griffiths and Jensen
\cite{Gr,JensenJDG,JM}.
We prove that a generic isotropic curve is deformable of order four and rigid to order five. %(Theorem \ref{thm:def-rig}).  
%
%The ratio ${\mathfrak d}^2(\delta)^2/\delta$ is a meromorphic function, the bending of the curve. 

Section \ref{s:cc isotropic curves} is devoted to the study of isotropic curves with constant bending.
%where the bending of an isotropic curve
%is the meromorphic function ${\mathfrak d}^2(\delta)^2/\delta$. 
Referring to the classical notion of $W$-curve in projective 3-space (cf. \cite{RB,Ch-S4}), we introduce the notion of isotropic $W$-curves
and describe their tamed 
surfaces. 
In the main theorem of the section, Theorem \ref{thm:cpt-cc}, we prove that a compact isotropic curve with constant bending 
%is dominated by an isotropic $W$-curve.
is a branched reparametrization of an isotropic $W$-curve.
To better illustrate the connections
%interrelations 
between the theory of isotropic curves and surface geometry, a number of examples is considered throughout the paper.
For figures
%the graphic part and for 
and for some symbolic computations, we used the software \textsl{Mathematica}.

\section{Preliminaries}\label{s:pre}

\subsection{The complex symplectic group}\label{ss:cpx-sym-gr}
Consider $\C^4$ with the symplectic form $\omega=dz^1\wedge dz^3+dz^2\wedge dz^4$.
Let $\mathrm{Sp}(2,\C)$ be the symplectic group of $\omega$ and $\mathfrak{sp}(2,\C)$ its Lie algebra. 
Denote by $({\bf e}_1,\dots,{\bf e}_4)$ the natural basis of $\C^4$ and by $A_j:\mathrm{Sp}(2,\C)\to \C^4$ the 
holomorphic map taking $A\in \mathrm{Sp}(2,\C)$ to $A{\bf e}_j$, $j=1,\dots,4$. The Maurer--Cartan form $\varphi$  of $\mathrm{Sp}(2,\C)$ 
is the $\mathfrak{sp}(2,\C)$-valued holomorphic 1-form such that $dA_i = \varphi^j_i A_j$, $i=1,\dots,4$.\footnote{Summation over 
repeated indices is assumed.} 
Let  ${\mathfrak C}^5$ be the complex vector space spanned by the skew-symmetric matrices
$$\mbox{\small $
{\rm L}_1={\bf e}_2^1-{\bf e}_1^2,\,\, {\rm L}_2={\bf e}_4^1-{\bf e}_1^4,\,\,
{\rm L}_3=\frac{1}{\sqrt{2}}({\bf e}_3^1-{\bf e}_1^3 -{\bf e}_4^2 +{\bf e}_2^4),\,\,
{\rm L}_4={\bf e}_3^2-{\bf e}_2^3,\,\, {\rm L}_5={\bf e}_4^3-{\bf e}_3^4, $}
$$
where ${\bf e}_i^j$ denotes the elementary $4\times 4$ matrix with 1
in the $(i,j)$ place and $0$ elsewhere, $i,j =1,\dots,4$.
On ${\mathfrak C}^5$, we consider the nondegenerate bilinear form 
$g_{\mathfrak C}(X,Y)=\frac{1}{2}{\rm tr}(JXJY)$, where $J={\bf e}_1^3+{\bf e}_2^4-{\bf e}_3^1-{\bf e}_4^2$.  
Then, $({\rm L}_1,\dots , {\rm L}_5)$ is a basis of ${\mathfrak C}^5$ such that  
%$$
%  (g_{\mathfrak C}({\rm L}_i,{\rm L}_j))_{1\le i,j\le 5}={\bf b}^1_5+{\bf b}^2_4+{\bf b}^3_3+{\bf b}^4_2+{\bf b}_1^5,
%   $$
$$
  (g_{\mathfrak C}({\rm L}_a,{\rm L}_b))_{1\le a,b\le 5}={\bf b}^5_1+{\bf b}^4_2+{\bf b}^3_3+{\bf b}^2_4+{\bf b}_5^1,
   $$
%where ${\bf b}^a_b$, $a,b01,\dots,5$, is the elementary $5\times 5$ matrice. 
where ${\bf b}_a^b$ denote the elementary $5\times 5$ matrices.
% with 1 in the $(a,b)$ place.
%and $0$ elsewhere.
%
For every $A\in \Sp(2,\C)$, the linear map ${\mathtt L}_{A}: {\mathfrak C}^5\ni X\mapsto A X \T\! A \in {\mathfrak C}^5$ 
is an orthogonal transformation
%$\forall \,A\in \Sp(2,\C)$ 
and ${\mathtt L}:\Sp(2,\C)\ni A \mapsto  {\mathtt L}_{A} \in \mathrm{O}({\mathfrak C}^5,g_{\mathfrak C})$ 
is a spin covering homomorphism. 
Let ${\mathfrak R}^5$ be the 5-dimensional real subspace of $\mathfrak{C}^5$ spanned by 
$$\mbox{\small $
{\rm E}_1=\frac{1}{\sqrt{2}}({\rm L}_1+{\rm L}_5),\,\, {\rm E}_2=\frac{i}{\sqrt{2}}({\rm L}_1-{\rm L}_5),\,\, {\rm E}_3={\rm L}_3,\,\,
{\rm E}_4=\frac{1}{\sqrt{2}}({\rm L}_2+{\rm L}_4),\,\, {\rm E}_5=\frac{i}{\sqrt{2}}({\rm L}_2-{\rm L}_4). $}
$$
Restricting  $g_{\mathfrak C}$ to ${\mathfrak R}^5$ we get a positive definite scalar product, denoted by $g_{{\mathfrak R}}$. 
By construction,
$({\rm E}_1,\dots,{\rm E}_5)$ is an orthogonal basis. The subspace ${\mathfrak R}^5$  is ${\rm Sp}(2)$-invariant and 
${\mathtt L}:  {\Sp}(2)\ni A\mapsto {\mathtt L}_{A}\in  \mathrm{O}({\mathfrak R}^5,g_{{\mathfrak R}})$ is a spin homomorphism.  
Consider ${\rm S}^{3}(\C^2)$, the third symmetric power of $\C^2$, and let ${\mathtt F}:{\rm S}^{3}(\C^2)\to \C^4$ 
be the isomorphism defined by
${\mathtt F}({\bf a}_{111})={\bf e}_1$, $ {\mathtt F}({\bf a}_{112})=-\sqrt[3]{2/9}{\bf e}_2$,  ${\mathtt F}({\bf a}_{222})=\sqrt[3]{1/6}{\bf e}_3$,
${\mathtt F}({\bf a}_{122})={\bf e}_4$,
where $({\bf a}_1,{\bf a}_2)$ is the natural basis of $\C^2$ and ${\bf a}_{ijk}$ is the symmetric tensor product 
${\bf a}_j{\bf a}_j{\bf a}_k$. 
Let $\widehat{{\mathtt S}}$ be the 
%standard 
representation of ${\rm SL}(2,\C)$ on ${\rm S}^{3}(\C^2)$
induced by the standard representation of $\SL(2,\C)$ on $\C^2$.
Then
${\mathtt S}:\mathrm{SL}(2,\C)\ni X\mapsto  {\mathtt F}\circ \widehat{{\mathtt S}}(X)\circ  {\mathtt F}^{-1}\in \Sp(2,\C)$  is a Lie group monomorphism. 
The subgroup ${\mathtt S}( \mathrm{SL}(2,\C))$ will be denoted by $\mathrm{H}$.

\subsection{The conformal structure of $\Q_3$}

Let $\Q_3$ be the compact complex 3-fold of Lagrangian 2-planes of
$\C^4$. The map $\pi: \Sp(2,\C)\ni A\mapsto [A_1\wedge A_2]\in \Q_3$ is a holomorphic principal bundle with structure group 
$$\mathrm{G}_0=\{A\in \Sp(2,\C) \mid [A_1\wedge A_2]=[{\bf e}_1\wedge {\bf e}_2]\}.$$  
Let $\TT=\{(P,{\bf v})\in \Q_3\times \C^4 \mid {\bf v}\in P\}$ be the tautological bundle of $\Q_3$ and  $\TTT$ 
the quotient bundle $(\Q_3\times \C^4)/\TT$.  If  $A:U\subset \Q_3\to \Sp(2,\C)$ is a holomorphic cross section of $\pi$, then 
$(\varphi_1^3,\varphi_1^4,\varphi_2^4)$ is a holomorphic coframe and $A_1\wedge A_2$ is a 
trivialization of $\bigwedge^2(\TT)$.\footnote{By slightly abusing notation, we will simply write $\varphi^i_j$ instead of $A^*(\varphi^i_j)$.} 
For every ${\bf v}\in \C^4$ and $P\in U$, we denote by $|[{\bf v}]|_P$ the equivalence class of ${\bf v}$ in $\TTT|_P$. Then 
$U \ni P\mapsto |[A_3|_P]|_P\wedge |[A_4|_P]|_P$ is a local trivialization of $\bigwedge^2(\TTT)$. Let  $(A_1\wedge A_2)^*$ be the dual  of $A_1\wedge A_2$. Then,
$${\bf g}_{|U}=(\varphi^3_1\varphi^4_2-(\varphi_1^4)^2)\otimes (A^1\wedge A^2)^*\otimes (|[A_3]|\wedge |[A_4]|))$$
is independent of the choice of $A$ and defines a nondegenerate holomorphic section ${\bf g}$ of $\mathrm{S}^{(2,0)}(\Q_3)\otimes \bigwedge^2(\TT)^*\otimes \bigwedge^2(\TTT)$. Then ${\bf g}$ determines a holomorphic $\Sp(2,\C)$-invariant conformal structure 
on $\Q^3$ (cf. \cite{KO1982,LeBrun82}).

\subsubsection{The Pl\"ucker  map} 

Let ${\bf x}\wedge {\bf y}\in \bigwedge^2(\C^4)$ be a Lagrangian bivector. Then  the $4\times 4$ skew-symmetric matrix  
${\mathtt l}({\bf x}\wedge {\bf y})={\bf x} \, \T{\bf y}-{\bf y} \, \T{\bf x}$ belongs to the nullcone of ${\mathfrak C}^5$. 
The {\it Pl\"ucker map}
$\lambda :  \Q_3\ni [{\bf x}\wedge {\bf y}]\mapsto [{\mathtt l}({\bf x}\wedge {\bf y})] \in {\mathbb P}\mathfrak{C}^5$ 
is a conformal biholomorphism of $\Q_3$ onto the null quadric ${\mathcal Q}_3$ of $ {\mathbb P}\mathfrak{C}^5$. 
The components of ${\mathtt l}$  with respect to $({\rm L}_1,\dots ,{\rm L}_5)$ will be denoted by ${\mathtt l}^a$, $a=1,\dots ,5$.
%$$[(x^1y^2-x^2y^1){\rm L}_1+(x^1y^4-x^4y^1){\rm L}_2+\sqrt{2}(x^1y^3-x^3y^1){\rm L}_3+(x^2y^3-x^3y^2){\rm L}_4+(x^3y^4-x^4y^3){\rm L}_5].$$

\subsubsection{Affine and unimodular conformal charts} 

On $\C^3$ with coordinates $(w^1,w^2,w^3)$, consider the quadratic form $g_{\C^3}=dw^1dw^3-(dw^2)^2$. 
Let ${\mathcal U}\subset \Q_3$ be the complement of the hyperplane section 
${\mathcal A}=\{[{\bf x}\wedge {\bf y}]\in \Q_3 \mid {\mathtt l}^1({\bf x}\wedge {\bf y})=0\}$. Then
$$
   \jmath: {\mathcal U}\ni [{\bf x}\wedge {\bf y}]\mapsto \frac{1}{{\mathtt l}^1({\bf x}\wedge {\bf y})}\, \left(-{\mathtt l}^4({\bf x}\wedge {\bf y}),
 {\mathtt l}^3({\bf x}\wedge {\bf y}),{\mathtt l}^2({\bf x}\wedge {\bf y})\right)\in \C^3
    $$
is a conformal biholomorphism, the {\it affine conformal chart} of $\Q_3$.
%
%The inverse is given by\footnote{since $[\vec{v}\wedge \vec{w}]$ is Lagrangian, then $v^1w^3-v^3w^1=v^4w^2-v^2w^4$}
%$$[\vec{v}\wedge \vec{w}]\in {\mathcal U}\to  \left(\frac{v^3 w^2-v^2 w^3}{v^1 w^2-v^2 w^1}, \frac{v^4 w^2-v^2 w^4}{v^1 w^2-v^2 w^1}, 
%\frac{v^1 w^4-v^4 w^1}{v^1 w^2-v^2 w^1}\right)\in \C^3.
%$$
%The affine conformal chart brakes the conformal structure. The residual symmetry group is 
%${\mathrm H}_{\C^3}=\{X\in \Sp(2,\C)\, /X[{\bf e}_3\wedge {\bf e}_4]=[{\bf e}%_3\wedge {\bf e}_4]\}$, a double 
%cover of the homothetic group of $(\C^3,g_{\C^3})$. The subgroup $\hat{{\mathrm H}}_{\C^3}=\{X\in \Sp(2,\C)\, /X({\bf e}_3\wedge {\bf %e}%_4)={\bf e}_3\wedge {\bf e}_4\}$ is the $2:1$ covering of the isometry group of $(\C^3,g_{\C^3})$. 
%
On $\SL(2,\C)$, consider the bi-invariant quadratic form $\widehat{g}=dx_1^1dx_2^2-dx^2_1dx^1_2$. 
Let $\widetilde{{\mathcal U}}\subset \Q_3$ be the complement of the hyperplane section 
${\mathcal B}=\{[{\bf x}\wedge {\bf y}]\in \Q_3 \mid {\mathtt l}^3({\bf x}\wedge {\bf y})=0\}$. 
The map
$\tilde{\jmath}: \widetilde{{\mathcal U}} \to {\rm SL}(2,\C)$ defined by
$$
 \tilde{\jmath}([{\bf x}\wedge {\bf y}])= \frac{\sqrt{2}}{{\mathtt l}^3({\bf x}\wedge {\bf y})}
\left({\mathtt l}^1({\bf x}\wedge {\bf y}){\bf a}_1^1-{\mathtt l}^2({\bf x}\wedge {\bf y}){\bf a}_1^2
-{\mathtt l}^4({\bf x}\wedge {\bf y}){\bf a}_2^1-{\mathtt l}^5({\bf x}\wedge {\bf y}){\bf a}_2^2\right),
$$
where ${\bf a}_r^s$, $r,s =1,2$, denote the elementary $2\times 2$ matrices,
is a conformal biholomorphism, the {\it unimodular conformal chart} of $\Q_3$. 
%As above, this chart brakes the confromal structure and the residual symmetry group is 
%${\rm H}_{{\rm SL}(2,\C)}$, isomorphic to ${\rm SL}(2,\C)\times {\rm SL}(2,\C)$, the double covering of $O(4,\C)$.

\subsection{The contact structure of $\CP^3$}
On $\CP^3$ we consider the $\Sp(2,\C)$-invariant contact structure induced  by $\zeta=-z^3dz^1-z^1dz^3+z^4dz^2-z^2dz^4$.  

\subsubsection{The unimodular contact chart}  

On ${\rm PSL}(2,\C)$, consider the left-invariant contact form 
$\widetilde{\zeta} = x^2_2dx^1_1-x^1_1dx^2_2+x^2_1 dx^1_2 - x_2^1dx^2_1$. Let ${\mathcal V}\subset \CP^3$ be the complement of 
the 2-dimensional quadric ${\mathcal Q}_2=\{[\xi]\in \CP^3 \mid \xi^1 \xi^3-\xi^2 \xi^4=0\}$. The  map
$$
 \jmath^{\sharp}: {\mathcal V}\ni [\xi]\mapsto  \left[\left(\xi^2{\bf a}_1^1 +\xi^1{\bf a}_1^2+\xi^3{\bf a}_2^1+\xi^4{\bf a}_2^2
  \right) \right]  \in {\rm PSL}(2,\C)
   $$
%$$\widehat{j}:[\xi]\in {\mathcal V}\to  \left[ \frac{1}{\sqrt{\xi^2\xi^4-\xi^1\xi^3}}
%\begin{pmatrix}\xi^2&\xi^1\\
%\xi^3&\xi^4 \end{pmatrix} \right]\in {\rm PSL}(2,\C)$$
is a contact biholomorphism,  the {\it unimodular contact chart} of $\CP^3$. 
%The unimodular contact chart brakes the contact structure of $\CP^3$ and the residual symmetry group is ${\rm H}_{{\rm PSL}}(2,\C)=\{X\in \Sp(2,\C)\,/\, X({\bf e}%_1\wedge {\bf e}_4)={\bf e}_1\wedge {\bf e}_4, 
%X({\bf e}_2\wedge {\bf e}_3)={\bf e}_3\wedge {\bf e}_3\}$, a subgroup isomorphic to ${\rm SL}(2,\C)$.

%The map
%$$\iota : (x^i_j)\in \SL(2,\C)/\{\pm {\rm Id}_{2\times 2}\}\to [^t(x^1_2,x^1_1,x^2_1,x^2_2)]\in \CP^3$$
%is a contact embedding onto the complement of the quadric $z^1z^3-z^2z^4=0$. The inverse is
%$$[^t(z^1,z^2,z^3,z^4)]\in \CP^3\to \left[ \begin{pmatrix}
%\frac{z^2}{\sqrt{z^2z^4-z^1z^3}}&\frac{z^1}{\sqrt{z^2z^4-z^1z^3}}\\
%\frac{z^3}{\sqrt{z^2z^4-z^1z^3}}&\frac{z^4}{\sqrt{z^2z^4-z^1z^3}}
%\end{pmatrix} \right]\in \SL(2,\C)/\{\pm {\rm Id}_{2\times 2}\}$$.

\subsection{Fibrations over Riemannian and Lorentzian space forms} 

\subsubsection{Hyperbolic and de~Sitter projections}

Let ${\mathfrak h}(2,\C)$ be the space of $2\times 2$ hermitian matrices with scalar product 
${2}(\alpha,\beta)={\det}(\alpha)+{\det}(\beta) - {\det}(\alpha+\beta)$. The hyperbolic and de~Sitter spaces can be 
realized as ${\mathcal H}^3=\{\alpha\in {\mathfrak h}(2,\C) \mid {\det}(\alpha)=1,\, {\mathrm{tr} (\alpha) >0}\}$ and 
${\mathcal H}^{1,2}=\{\alpha\in {\mathfrak h}(2,\C) \mid{\det}(\alpha)=-1\}$, equipped with the Riemannian and Lorentzian 
structures inherited from  $(\cdot\,, \cdot)$. 
%The maps $\pi_{{\mathcal H}^{3}}: {\rm PSL}(2,\C)\ni [B]\mapsto  B\cdot  \T \bar{B}\in  {\mathcal H}^{3}$ and 
%$\pi_{{\mathcal H}^{1,2}}: {\rm PSL}(2,\C)\ni [B]\mapsto B\cdot({\bf a}_1^1-{\bf a}_2^2)\cdot  \T\bar{B}\in {\mathcal H}^{1,2}$ 
%are the bundles of the oriented (resp. oriented and future-directed) orthonormal frames of ${\mathcal H}^{3}$ and ${\mathcal H}^{1,2}$ respectively. 
%
The projections $\pi_{{\mathcal H}^{3}}: {\rm PSL}(2,\C)\ni [B]\mapsto  B \, \T\bar{B}\in  {\mathcal H}^{3}$ and 
$\pi_{{\mathcal H}^{1,2}}: {\rm PSL}(2,\C)\ni [B]\mapsto B ({\bf a}_1^1-{\bf a}_2^2)\,  \T\bar{B}\in {\mathcal H}^{1,2}$ 
are, respectively, the bundle of oriented orthonormal frames of ${\mathcal H}^{3}$ 
and that of oriented and future-directed orthonormal frames of ${\mathcal H}^{1,2}$. 
%
%The map $\pi_{{\mathcal H}^{3}}: {\rm PSL}(2,\C)\ni [B]\mapsto  B\cdot  \T \bar{B}\in  {\mathcal H}^{3}$, respectively
%$\pi_{{\mathcal H}^{1,2}}: {\rm PSL}(2,\C)\ni [B]\mapsto B\cdot({\bf a}_1^1-{\bf a}_2^2)\cdot  \T\bar{B}\in {\mathcal H}^{1,2}$,
%is the bundle of the oriented orthonormal frames of ${\mathcal H}^{3}$, respectively of oriented and future-directed 
%orthonormal frames of  ${\mathcal H}^{1,2}$. 
%
In turn,
\[
\begin{cases}
\tilde{\ell}_{{\mathcal H}^{3}} = \pi_{{\mathcal H}^3}\circ [\tilde{\jmath}]:\widetilde{{\mathcal U}}\to {\mathcal H}^3,\\
\tilde{\ell}_{{\mathcal H}^{1,2}}=  \pi_{{\mathcal H}^{2,1}}\circ [\tilde{\jmath}]:\widetilde{{\mathcal U}}\to {\mathcal H}^{1,2},\\
{\ell}^{\sharp}_{{\mathcal H}^{3}} = \pi_{{\mathcal H}^3}\circ \jmath^{\sharp}: {\mathcal V}\to {\mathcal H}^3,\\ 
{\ell}^{\sharp}_{{\mathcal H}^{1,2}}=  \pi_{{\mathcal H}^{2,1}}\circ \jmath^{\sharp}: {\mathcal V}\to {\mathcal H}^{1,2}\end{cases}\] give on $\widetilde{{\mathcal U}}\subset \Q_3$ and
${\mathcal V}\subset \CP^3$ structures of principal bundles either on ${\mathcal H}^3$ or on  ${\mathcal H}^{1,2}$. 
The structure groups are ${\rm SU}(2)$ or ${\rm SU}(1,1)$, respectively.

\subsubsection{Euclidean and Minkowski projections} 
Let $\R^3$ and $\R^{1,2}$ denote Euclidean and Minkowski 3-space, respectively. The maps 
 \[\begin{cases}
 \pi_{\R^3}:\C^3 \ni (z_1,z_2,z_3) \mapsto \big(\frac{{\rm Re}(z_1+z_3)}{2}, \frac{{\rm Im}(z_1-z_3)}{2},{\rm Re}(z_2)\big)\in \R^3,\\
  \pi_{\R^{1,2}}:  \C^3 \ni (z_1,z_2,z_3)\mapsto \big({\rm Re}(z_2),\frac{{\rm Im}(z_1+z_3)}{2}, \frac{{\rm Im}(z_1-z_3)}{2}\big)\in \R^{1,2}
\end{cases}\]
\noindent give on $\C^3$ two structures of real vector bundle. Consequently, $\ell_{\R^3}=\pi_{\R^3}\circ \jmath$ 
and $\ell_{\R^{1,2}}=\pi_{\R^{1,2}}\circ \jmath$ make ${\mathcal U}$ a real vector bundle over $\R^3$ or over $\R^{1,2}$.

\subsubsection{The twistor fibration} 

A 3-dimensional linear subspace ${\mathfrak p}\subset {\mathfrak C}^5$ is said to be {\it parabolic} if 
${\dim}\mathrm{Ker}(g_{\mathfrak C}|_{{\mathfrak p}\times {\mathfrak p}})=2$. 
The totality of all parabolic subspaces, denoted by ${\mathfrak P}^3$, is a
compact complex $3$-fold acted upon transitively both by $\Sp(2,\C)$ and ${\rm Sp}(2)$. 
If $[\xi]\in \CP^3$, then ${\mathfrak p}_{\xi}=\{X\in {\mathfrak C}^5 \mid
{XJ}( \xi\,\T\xi) -  (\xi\, \T\xi) JX=0\}$ belongs to ${\mathfrak P}^3$. 
The map ${\mathfrak p}: \CP^3 \ni \xi\mapsto {\mathfrak p}_{\xi}\in {\mathfrak P}^3$ is an equivariant biholomorphism. 
Let ${S}^4$  be the unit sphere of ${\mathfrak R}^5$. Consider the $\mathrm{SU}(2)$-bundle 
$\pi_{{\mathfrak P}}: {\Sp}(2)\ni A\mapsto [{\mathtt L}(A)_1\wedge {\mathtt L}(A)_2\wedge {\mathtt L}(A)_3]\in {\mathfrak P}^3$.
By construction, the map ${\rm Sp}(2)\ni A\mapsto {\mathtt L}(A)_3\in S^4$ is constant along the fibers of $\pi_{{\mathfrak P}}$,
and hence it descends to a map ${\mathfrak w}: {\mathfrak P}^3\to {S}^4$ which, upon the identification of 
$\CP^3$ with ${\mathfrak P}^3$, amounts to  the {\it twistor fibration} (cf. \cite{Br-S4}). 
%It is useful to note that $[{\mathfrak w}(\mathfrak p)]_{\R}= \mathfrak p\cap {\mathfrak R}^5$.

\section{Projective structures}\label{s:proj-struct}

We will briefly recall the notion of a complex projective structure on a Riemann surface. We will adapt to this specific 
case the general definition given by S. Kobayashi in \cite{K1}, which goes back to E. Cartan \cite{Ca5}.
 
 \begin{defn}\label{d:prj-str}   
 Let $S$ be a Riemann surface and let $\mathrm{SL}(2,\C)_1$ be the group of upper triangular $2\times 2$ 
 unimodular matrices. 
 A {\it complex projective structure} on $S$ is a holomorphic principal $\mathrm{SL}(2,\C)_1$-bundle 
 $P\to S$, equipped with a holomorphic Cartan connection $\eta$.
\end{defn}

\begin{lemma}\label{l:proj-1}
%Let $(P,\eta)$ be a projective structure. Then, for every $p_0\in S$ there exist a chart $(U,z)$ on an open 
%neighborhood of $p_0$ and a cross section ${\mathfrak p}:U\to P$ such that ${\mathfrak p}^*(\eta)={\bf a}^1_2dz$.
Let $(P,\eta)$ be a projective structure. About any point $p_0\in S$, there exist a coordinate chart $(U,z)$ 
and a cross section ${\mathfrak p}:U\to P$, such that ${\mathfrak p}^*(\eta)={\bf a}^1_2dz$.
\end{lemma}

\begin{proof} 
Let ${\mathfrak q}:V\to P$ be a section of $P$, defined on a neighborhood of $p_0$. Possibly shrinking $V$, 
there is a coordinate $w:V\to \C$, such that
${\mathfrak q}^*(\eta)=(x^1_1({\bf a}_1^1-{\bf a}_2^2)+{\bf a}_2^1+x^1_2{\bf a}_1^2)dw$.  
By the existence and uniqueness theorem for holomorphic ODE (see for instance \cite[p. 46]{H}), there exists a 
holomorphic function $g:U\to \C$, defined on a smaller neighborhood  $U\subset V$ of $p_0$, such that 
$g''-g^2+(x^1_1)'+(x^1_1)^2+x^1_1=0$, where $h'$ denotes the derivative of a function $h$ with respect to $dw$. 
Let $b=e^{-g}(x^1_1+g')$ and define $z:U\to \C$ by $dz=e^{2g}dw$. Possibly shrinking $U$, $z$ is a coordinate. Put 
${\mathfrak p}={\mathfrak q}\cdot (e^g{\bf a}_1^1+e^{-g}{\bf a}_2^2+ b{\bf a}_1^2)$. Then, $(U,z)$ and ${\mathfrak p}$ 
satisfy the required condition.
\end{proof}

\begin{defn} 
A chart $(U,z)$ is said to be {\it adapted} to $(P,\eta)$ if there exists a cross section ${\mathfrak p}:U\to P$, 
such that ${\mathfrak p}^*(\eta)={\bf a}_2^1 dz$.  We call ${\mathfrak p}$ a {\it flat section} of $P$. By Lemma \ref{l:proj-1},  
%we may cover $S$ with 
$S$ can be covered by
an atlas $\mathcal P$ of adapted charts.  An atlas on $S$ is said to be {\it projective} if its transition functions 
are restrictions of M\"obius transformations.
\end{defn}

\begin{lemma}\label{l:proj-2}
The atlas of the charts adapted to  $(P,\eta)$ is projective.
\end{lemma}

\begin{proof}  
It suffices to prove that the transition function between two adapted local coordinates $z$ and $w$, defined on the same simply 
connected open neighborhood $U$, is the restriction of a M\"obius transformation.  Let ${\mathfrak p}, {\mathfrak q}:U\to P$ 
be the cross sections such that ${\mathfrak p}^*(\eta)={\bf a}_2^1 dz$ and ${\mathfrak q}^*(\eta)={\bf a}_2^1 dw$. Consider $x=e^g {\bf a}_1^1+e^{-g}{\bf a}_2^2+b {\bf a}_1^2 : U \to {\rm SL}(2,\C)_1$, such that ${\mathfrak q}={\mathfrak p}\cdot x$. Then,
$$
{\bf a}_2^1dw=e^{2g}dz \,{\bf a}^1_2+(dg-e^gbdz){\bf a}^1_1-(dg-e^gbdz){\bf a}^2_2+((e^{-g}(db+bdg)-b^2dz){\bf a}^2_1.$$
This implies $bdz=e^{-g}dg$ and $d(dg/dz)dz-(dg)^2=0$. From the second equation, it follows that $g=c_2-\log(z+c_1)$, 
where $c_1,c_2$ are two constants of integration. 
Then, $dw=e^{2g}dz$ yields  $w=(z+c_1)^{-1}(c_3z+(c_1c_3-e^{2c_2}))$, for some $ c_3\in \C$. This concludes the proof.
\end{proof}
 
 \begin{remark} 
In the literature \cite{LP2009,OvTa2005}, a projective structure on a Riemann surface $S$ is often defined in terms of 
a projective atlas. 
%This latter definition is apparently more general. However, 
It is not difficult to prove  that every projective atlas is 
originated by a projective structure in the sense of Definition \ref{d:prj-str}. 
\end{remark}
 
\noindent  Let $(P,\eta)$ be a projective structure on $S$ and let ${\mathcal P}$ be its the projective atlas. 
Let ${\mathcal M}^q$ be the sheaf of meromorphic differentials of order $q$. For each $(U,z)\in {\mathcal P}$,
we define\footnote{We implicitly assume that ${\mathfrak d}_ {(U,z)}(0)=0$.} 
${\mathfrak d}_{(U,z)} : {\mathcal M}^4|_U\to {\mathcal M}^2|_U$ by
 \begin{equation}\label{secondorderoperator}
  {\mathfrak d}_ {(U,z)}({\mathrm Z} dz^4)=\big(\frac{1}{2}\frac{{\mathrm Z}''}{{\mathrm Z}}-\frac{9}{16}\frac{{\mathrm Z}'^2}{{\mathrm Z}^2}\big)dz^2
   \end{equation}
 
 \begin{lemma}\label{l:proj-3}
 The second order operator ${\mathfrak d}_ {(U,z)}$ is independent of the choice of the projective chart.
 \end{lemma}
 
 \begin{proof} 
 Let ${\tilde z}$ and $z$ be two coordinates on the neighborhood $U$, with transition function $h={\tilde z}\circ z^{-1}$,
  and let ${\mathfrak d}_ {(U,z)}$,  ${\mathfrak d}_ {(U,{\tilde z})}$ be defined as in \eqref{secondorderoperator}.  
  It is a computational matter to check that
%  \begin{equation}\label{ftrans}
%    {\mathfrak d}_ {(U,z)}={\mathfrak d}_ {(U,{\tilde z})}+2z^*(\mathcal{S}(h)),
%       \end{equation} 
         \begin{equation}\label{ftrans}
    {\mathfrak d}_ {(U,z)}={\mathfrak d}_ {(U,{\tilde z})}+2\,{\mathcal{S}_z(h) dz^2},
       \end{equation} 
 where 
 %$\mathcal{S}_z(h)$ 
$\mathcal{S}_z(h) = (\frac{h''}{h'})' -\frac12(\frac{h''}{h'})^2$
 is the Schwarzian derivative of $h$ with respect to $z$. If both charts are projective, 
 then $h$ is a M\"obius transformation, which implies
 $\mathcal{S}_z(h)=0$, and hence the result.\end{proof}
 
 \begin{defn} 
 Let $(P,\eta)$ be a projective structure on $S$ and let ${\mathcal P}$ be its projective atlas. Then, there exists  
 a second order differential operator ${\mathfrak d} : {\mathcal M}^4\to {\mathcal M}^2$, such 
 that ${\mathfrak d}|_U={\mathfrak d}_{(U,z)}$, for every $(U,z)\in {\mathcal P}$.
\end{defn}

%\begin{lemma}  Let  $(U,z)$ be any complex chart and ${\mathfrak p}:U\to P$ be a cross section. Put $\epsilon^i_jdz={\mathfrak p}^*(\eta^i_j)$. %Then
%\begin{equation}\label{first}{\mathfrak d} |_U= {\mathfrak d}_ {(U,z)} + {\mathfrak r}_{(U,z,{\mathfrak p})}dz^2.\end{equation}
%where
%\begin{equation}\label{addterm}{\mathfrak r}_{(U,z,{\mathfrak p})}=-2\frac{(\epsilon^2_1)''}{\epsilon^2_1}+\frac{(\epsilon^2_1) '}{\epsilon^2_1}\left( %3\frac{(\epsilon^2_1) '}{\epsilon^2_1} -4\epsilon^1_1\right)+4\left((\epsilon^1_1)'+4(\epsilon^1_1)^2+\epsilon^1_2\epsilon^2_1 \right).
%\end{equation}
%\end{lemma}
%\begin{proof} It is easy to see that ${\mathfrak r}_{(U,z,{\mathfrak p})}dz^2$ doesn't depend on the choice of the section ${\mathfrak p}$.  
% Let ${\widetilde z}$ be another coordinate and $h$ be the transition function ${\widetilde z}\circ z^{-1}$. A straightforward computation gives 
% ${\mathfrak r}_{(U,z,{\mathfrak p})}dz^2={\mathfrak r}_{(U,{\widetilde z},{\mathfrak p})}d{\widetilde z}^2-2z^*({\mathcal S}(h))$.
% From (\ref{ftrans}) it follows that the right hand side of (\ref{first}) doesn't depend on the choice of the local coordinate and of the cross section.
 % If $(U,z)$ is a projective chart and ${\mathfrak p}$ is a flat section, then ${\mathfrak r}_{(U,z,{\mathfrak p})}=0$. This implies the result.
%\end{proof}

To apply projective structures in the study of isotropic curves we need a slightly more general notion.

\begin{defn}
Let ${\rm D}\subset S$ be a discrete set. A {\it meromorphic projective structure} on $S$ is a projective structure 
$(P,\eta)$ on $S\setminus D$, satisfying the following condition:  for every $p_0\in D$, there exist an open 
neighborhood $U$, with $U\cap D=\{p_0\}$, and a section ${\mathfrak p}:U\setminus \{p_0\}\to P$, such 
that ${\mathfrak p}^*(\eta)$ is meromorphic on $U$. A point $p_0\in D$ is a {\it removable singularity}
if $P$ and $\eta$ can be extended across $p_0$. If the points of $D$ are not removable singularities, 
we say that $D$ is the \textit{singular locus} of $(P,\eta)$.
We can use \eqref{ftrans} to define the 
operator ${\mathfrak d} : {\mathcal M}^4\to {\mathcal M}^2$ also for meromorphic projective structures.
\end{defn}

\section{Isotropic curves in $\Q_3$}\label{s: isotropic curves}

 \subsection{Isotropic curves and Legendre associates}  
A nonconstant holomorphic map $f : S \to \Q_3$  from a connected Riemann surface $S$ into the complex 
quadric $\Q_3$  is said to be an \textit{isotropic curve} if $f^*({\bf g})=0$.  Given $[\xi]\in \CP^3$, the pencil 
$\{P\in \Q_3\ \mid [\xi]\subset P\}$
 is a complex null geodesic \cite{LeBrun82}, referred to as a {\it null ray}. We will only consider isotropic curves 
 which are not null rays.

Let $f$ be an isotropic curve and $p_0$ a point of $S$. Let $(U,z)$ be a complex chart centered at $p_0$ and
${\bf u}_1$, ${\bf u}_2:U\to \C^4$ be two maps, such that  $f|_U=[{\bf u}_1\wedge {\bf u}_2]$. Let ${\mathfrak s}(2,\C)$ 
denote the space of symmetric $2\times 2$ 
 matrices and consider the nonconstant map ${\dot{\bf m}}=(\dot{m}_{ij}):U\to {\mathfrak s}(2,\C)$, 
 defined by $\dot{m}_{ij}=\omega({\bf u}_i,{\bf u}_j')$. 
 Notice that $f$ is isotropic if and only if  ${\det}(\dot{\bf m})=0$. Possibly shrinking $U$, there exist a nowhere zero map 
 $\dot{\mathtt m}:U\to  {\mathfrak s}(2,\C)$ and a nonnegative integer $k_1$, such that $\dot{\bf m}=z^{k_1}\dot{\mathtt m}$. 
The integer $k_1$ is the ramification index of $f$ at $p_0$. Possibly switching ${\bf u}_1$ and ${\bf u}_2$ and shrinking $U$, 
we may assume that $\dot{\mathtt m}_{22}$ is nowhere zero. Then, $[-\dot{\mathtt m}_{22}{\bf u}_1+\dot{\mathtt m}_{12}{\bf u}_2]:U\to \CP^3$ 
does not depend on the choice of ${\bf u}_1$ and ${\bf u}_2$. Thus there exists a holomorphic map $f^{\sharp} : S\to \CP^3$, 
such that $f^{\sharp}|_U=[-\dot{\mathtt m}_{22}{\bf u}_1+\dot{\mathtt m}_{12}{\bf u}_2]$. By construction, $f^{\sharp}$ is a Legendre map and $f^{\sharp}(S)$ is not contained in any contact line of $\CP^3$.  We call $f^{\sharp}$ the {\it Legendre associate} of $f$. 
The divisors of the critical points of $f$ and $f^{\sharp}$ will be denoted by $\Delta_f$ and $\Delta_{f^\sharp}$, respectively.
 
 \begin{remark} 
 Conversely, if $f^{\sharp}:S\to \CP^3$ is a Legendrian curve not contained in a contact line, 
 then $f=[f^{\sharp}	\wedge 	f^{\sharp}\,']:S\to \Q_3$ is an isotropic curve which is not contained in any isotropic ray. 
 This and the result of Bryant \cite[Theorem G]{Br-S4},
asserting that any compact connected Riemann surface can be holomorphically
embedded in $\CP^3$ as a Legendrian curve, imply that for every compact Riemann surfaces 
$S$ there exists a generically one-to-one isotropic curve $f : S \to \Q_3$. 
\end{remark}

 \subsection{Isotropic curves and classical surface theory}  
For a given isotropic curve $f:S\to \Q_3$, consider  the discrete subsets ${\rm E}_f=f(S)\cap {\mathcal A}$, 
$\widetilde{{\rm E}}_f=f(S)\cap {\mathcal B}$, and ${\rm E}^{\sharp}_f=f^{\sharp}(S)\cap {\mathcal Q}$. 
If $S$ is compact, these are finite subsets. 
%
%the discrete subsets ${\rm E}_f,\widetilde{{\rm E}}_f,{\rm E}^{\sharp}_f$ be the discrete  substet ${\rm E}_f=f(S)\cap {\mathcal A}$, $\widetilde{{\rm E}}_f=f(S)\cap {\mathcal B}$ and ${\rm E}^{\sharp}_f=f^{\sharp}(S)\cap {\mathcal Q}$. If $S$ is compact, they are finite. 
%We put $S_{\star}=S\setminus {\rm E}_f$, $\widetilde{S}=S\setminus \widetilde{{\rm E}}_f$ and $S^{\sharp}=S\setminus {\rm E}^{\sharp}_f$.  The following theorem is a compendium of results ranging from the nineteenth century to the first decade of this century  \cite{W1,W2,Br1987,Br-S4,Ch-S4,KO1983,KUY-O,Lee2005,GMM-MA}.
%
Let $S_{f}=S\setminus {\rm E}_f$, 
%$S_{\star}=S\setminus {\rm E}_f$, 
 $\widetilde{S}_f=S\setminus \widetilde{{\rm E}}_f$, 
 and $S^{\sharp}_f=S\setminus {\rm E}^{\sharp}_f$.  The following theorem is a compendium of results ranging from the 19th 
 century through the first decades of this century  \cite{W,Br1987,Br-S4,Ch-S4,KO1983,KUY-O,Lee2005,GMM-MA}.

\begin{thmx}\label{thm:associatedsurfaces}
Let $f:S\to \Q_3$ be an isotropic curve. Then,
\begin{enumerate}

\item  $\phi_{\R^3}:=\ell_{\R^3}\circ f :S_f\to \R^3$ is  a conformal, branched minimal immersion. 
The points of  ${\rm E}_f$ are the ends of $\phi_{\R^3}$.

\item $\phi_{\R^{1,2}}:=\ell_{\R^{1,2}}\circ f : S_f\to \R^{2,1}$ is a conformal, branched maximal immersion. 
The points of  ${\rm E}_f$ are the ends of $\phi_{\R^3}$.

\item $\widetilde{\phi}_{{\mathcal H}^3}:=\tilde{\ell}_{{\mathcal H}^3}\circ f: \widetilde{S}_f\to {\mathcal H}^3$ is a conformal, 
branched CMC 1 immersion.
The points of $\widetilde{{\rm E}}_f$  are the ends of $\widetilde{\phi}_{{\mathcal H}^3}$.

\item  $\widetilde{\phi}_{{\mathcal H}^{1,2}}:=\tilde{\ell}_{{\mathcal H}^{1,2}}\circ f :  \widetilde{S}_f\to {\mathcal H}^{1,2}$ 
is a conformal, branched spacelike CMC 1 immersion. The points of $\widetilde{{\rm E}}_f$  
are the ends of $\widetilde{\phi}_{{\mathcal H}^{1,2}}$.

\item  $\phi^{\sharp}_{{\mathcal H}^3}:=\ell^{\sharp}_{{\mathcal H}^{3}}\circ f^{\sharp}  : S^{\sharp}_f\to {\mathcal H}^{3}$ is a  flat front. The points of ${\rm E}^\sharp_f$ are the ends of $\phi^{\sharp}_{{\mathcal H}^3}$.

\item  $\phi^{\sharp}_{{\mathcal H}^{1,2}}:=\ell^{\sharp}_{{\mathcal H}^{1,2}}\circ f^\sharp  : {S}^\sharp_f\to {\mathcal H}^{1,2}$ is a spacelike flat front.
The points of ${\rm E}^\sharp_f$ are the ends of $\phi^{\sharp}_{{\mathcal H}^{1,2}}$.

\item $\phi^{\sharp}_{S^4}:={\mathfrak w}\circ f^{\sharp}:S\to S^4$ is a branched superminimal immersion.
\end{enumerate}
The listed maps are said to be the {\em branched (or frontal) immersions  tamed} by $f$.
\end{thmx}

\begin{remark}
According to  \cite[Theorem F]{Br-S4}, a Legendre curve from a Riemann surface $S$ to $\CP^3$ 
not contained in a contact line can be written in the form
\begin{equation}\label{brrep}
   f^{\sharp} =[(2g',2gg',2hg'-gh',h')],
      \end{equation}
where $g$ and $h$ are meromorphic functions with $g$ nonconstant. In view of Theorem \ref{thm:associatedsurfaces}, one can 
derive explicit representation formulae for the branched (frontal) immersions tamed by $f$ in terms of the meromorphic 
functions $g$ and $h$ 
(see for instance  \cite{Small1994,KUY-O,ET,GMM-MA,MN-Lincei}). 
As an example, in \cite{KUY-O}, the Bryant representation formula for Legendrian curves in $\CP^3$ was implicitly used to 
find flat fronts in ${\mathcal H}^3$ with $n$ smooth ends (see Figure \ref{FIGURA-1}). The generating meromorphic functions 
are $g_n(z)=z^{-1}(z^n-1)^{(2-n)/n}$ and $f_n(z)=z^{-2}(1+z^n)(z^n-1)^{(2-n)/n}$, $n\in {\mathbb N}$. 
They are defined on a covering of $\C$ punctured at the $n$th roots of the unity.
 \end{remark}

 \begin{figure}[h]
\begin{center}
\includegraphics[height=5.6cm,width=11.2cm]{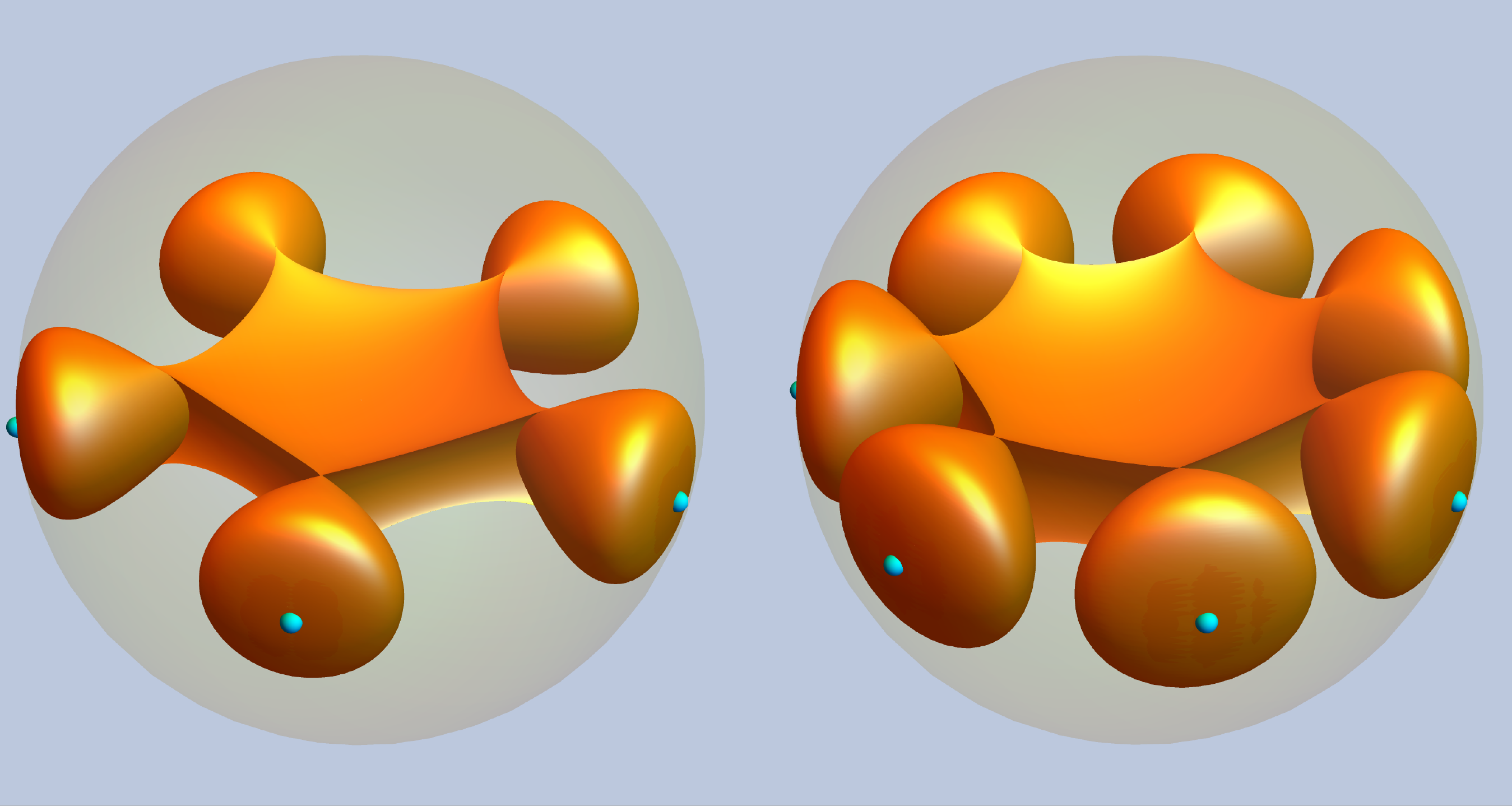}
\caption{\small{Flat fronts in hyperbolic 3-space with 5 and 7 ends, with rotational symmetries of order 5 and 7, 
originated by the meromorphic functions $(g_5,h_5)$ and $(g_7,h_7)$, respectively.}}\label{FIGURA-1}
\end{center}
\end{figure}

\begin{defn} 
Let $\phi$ and $\hat \phi$ be tamed  by  $f$ and  $\hat f$, respectively. We say that $\phi$ is a {\it conformal
Goursat transform} of $\hat \phi$
if there exists $A\in \Sp(2,\C)$, such that $f(S)=A\cdot \hat f(\hat S)$. 
\end{defn}

\begin{remark} 
The (\textit{classical}) \textit{Goursat transformation} 
was originally introduced for minimal surfaces in Euclidean space by Goursat \cite{Go,JMN}.
%is an old concept in minimal surface Theory. 
The classical definition can be rephrased as follows.
\vskip0.1cm

\noindent $\bullet$   Let $\phi$ and $\hat \phi$ be two branched minimal immersions in $\R^3$ tamed by $f$ and $\hat f$,
respectively; $\phi$ and $\hat \phi$ are (classical) Goursat transforms of each other if there exists an element $A$ of the subgroup
$\{X\in \Sp(2,\C)\mid X({\bf e}_1\wedge {\bf e}_2)={\bf e}_1\wedge {\bf e}_2,\, X({\bf e}_3\wedge {\bf e}_4)={\bf e}_3\wedge {\bf e}_4\}\cong {\rm SL}(2,\C)$, such that $\hat f=A\cdot f$.
\vskip0.1cm

\noindent This means that the Gauss maps of $\phi$ and $\hat \phi$ differ by a M\"obius transformation of $S^2$. 
Taking this point of view, one can define a Goursat transformation for isothermic surfaces in $\R^3$ \cite{HJH} which 
generalizes the classical one but, excluding the case of minimal surfaces, seems not directly related to our definition. 

A Goursat transformation for CMC 1 surfaces in ${\mathcal H}^3$ has been considered in
\cite{D,RUY,UY-Annals,Bo-Pe2005}. It can be described as follows.
\vskip0.1cm

\noindent $\bullet$  Let $\phi$ and $\hat \phi$ be two CMC 1 immersed surfaces of ${\mathcal H}^3$ tamed by $f$ 
and $\hat f$; $\phi$ and $\hat \phi$ are (\textit{hyperbolic}) {\it Goursat transforms} of each other if there exists an element $A$ of the subgroup
 $\{X\in \Sp(2,\C) \mid  X{\bf e}_1={\bf e}_1,\,  X{\bf e}_3={\bf e}_3,\, X({\bf e}_2\wedge {\bf e}_4)={\bf e}_2\wedge {\bf e}_4\}
\cong {\rm SL}(2,\C)$, such that $\hat f=A\cdot f$.
\vskip0.1cm

\noindent Therefore, our definition is a natural generalization of both the classical and hyperbolic Goursat transformations.

%\noindent The {\it Lawson correspondence}  between cmc-1 surfaces in ${\mathcal H}^3$ and minimal surfaces in $\R^3$ is equivariant with respect to %the classical and hyperbolic Goursat transforms. 
\end{remark}

\subsection{The projective structure of an isotropic curve} 

The zeroth order frame bundle along an isotropic curve $f : S \to \Q_3$ is the principal $\mathrm{G}_0$-bundle
$ \mathcal F_0 \to S$,
\begin{equation}\label{Zeroth}
 \mathcal F_0 = \left\{(p, A) \in S\times \Sp(2,\C) \mid f(p) = \pi(A) =[A_1 \wedge A_2] \right\}.
\end{equation}
The holomorphic sections of $ \mathcal F_0$ are called {\it symplectic frames} along $f$. Let $\mathrm{H}_1$ 
be the 2-dimensional Lie subgroup
$\{X\in {\rm H} \mid  (X {\bf e}_1)\wedge {\bf e}_1=0\}$, isomorphic to  $\mathrm{SL}(2,\C)_1$.  

\begin{thmx}\label{thm:FourthReduction} 
Let $f:S\to \Q_3$ be an isotropic curve and $S_{\star}=S\setminus |\Delta_f|\cup |\Delta_{f^{\sharp}}|$. There exists a 
unique 
%$\mathrm{H}_1$-reduction 
reduced bundle $\pi:\mathcal F\to S_{\star}$ of $ \mathcal F_0 |_{S_{\star}}$ with structure group $\mathrm{H}_1$
%to ${\rm H}_1$, 
such that:
\vskip0.1cm
\noindent $\bullet$  $ \mathcal F\ni (p,A)\mapsto A\in \Sp(2,\C)$ is an integral manifold of the Pfaffian differential system given by
\begin{equation}\label{PE4}
 \varphi^3_1=\varphi^4_1=\varphi^2_1-\varphi^4_2 = \varphi^1_1-3\varphi^2_2=4\varphi^1_2-3\varphi_4^2
       =\varphi_4^1=0,\quad \varphi_2^4   \neq 0.
\end{equation}

\vskip0.1cm
\noindent $\bullet$  There exists a meromorphic quartic differential  $\delta$ on $S$ 
 such that ${\pi}^*(\delta)=\varphi^1_3 \, (\varphi^4_2)^3$.

\vskip0.1cm
\noindent $\bullet$  $\mathcal F$ with the $\mathfrak{sl}(2,\C)$-valued $1$-form  $\eta=\varphi^2_2({\bf a}_1^1-{\bf a}_2^2)-6^{-1/3}\varphi^4_2{\bf a}_2^1-(3/32)^{1/3}\varphi^2_4{\bf a}_1^2$
is a meromorphic projective structure with singular locus $D\subseteq |\Delta_f|\cup |\Delta_{f^{\sharp}}| $.

\end{thmx}
\begin{proof}  
%We premise two lemmas :
%\vskip0.2cm
For the proof of Theorem \ref{thm:FourthReduction}, we need the following two lemmas.

%\noindent {\it Lemma 1}. 
\begin{lemma}\label{l:lemma1-thmB}
Let $f$ be an isotropic curve defined on an open disk ${\rm D}_{\epsilon}\subset \C$ centered at the origin. Let $k_1$ and $k_2$ be the ramification indices of $f$ and $f^{\sharp}$ at $0$. Possibly shrinking ${\rm D}_{\epsilon}$ there exists
a meromorphic lift $A:{\dot {\rm D}}_{\epsilon}\to \Sp(2,\C)$ of $f$, holomorphic on $\dot{{\rm D}}_{\epsilon}$ 
and  satisfying \eqref{PE4}. If $k_1=k_2=0$, $A$ is holomorphic.
\end{lemma}
%\vskip0.1cm

%\noindent{\it Proof of Lemma 1}.  
\begin{proof}[Proof of Lemma \ref{l:lemma1-thmB}]

Let $({\bf b}_1,\dots, {\bf b}_4)\in \Sp(2,\C)$, so that $f(0)=[{\bf b}_1\wedge {\bf b}_2]$.  Then,
$f=[{\bf u}_1\wedge {\bf u}_2]$, where ${\bf u}_1={\bf b}_1+m_{11}{\bf b}_3+m_{12}{\bf b}_4$ 
and ${\bf u}_2={\bf b}_2+m_{12}{\bf b}_3+m_{22}{\bf b}_4$. 
Let 
$m=m_{11}{\bf a}^1_1+m_{12}({\bf a}^1_2+{\bf a}^2_1)+m_{22}{\bf a}^2_2$ and $k_1$ be the ramification index of $f$ at $0$. 
Then, 
$dm=z^{k_1}\dot{{\mathtt m}}dz$, where  $\dot{{\mathtt m}}:{\rm D}_{\epsilon}\to  {\mathfrak s}(2,\C)$  
satisfies $\dot{{\mathtt m}}(0)\neq 0$ and ${\rm det}(\dot{{\mathtt m}})=0$. 
Without loss of 
generality, we assume that $\dot{{\mathtt m}}_{22}(z)\neq 0$, $\forall z$. Let $X_0\in \Sp(2,\C)$ be given by
%
%\textcolor{red}
{$X_0=(-\dot{{\mathtt m}}_{22}{\bf e}_1^1-\dot{{\mathtt m}}_{22}^{-1}{\bf e}_3^3+\dot{{\mathtt m}}_{12}{\bf e}_2^1+
\dot{{\mathtt m}}_{22}^{-1}\dot{{\mathtt m}}_{12}{\bf e}_3^4+{\bf e}_2^2+{\bf e}_4^4)_{|0}$}
and let $B:{\rm D}_{\epsilon}\to \Sp(2,\C)$ be the frame $B=({\bf u}_1,{\bf u}_2,{\bf b}_3,{\bf b}_4)\, X_0$. Then, 
$B^{-1}dB=z^{k_1}
r(a^2{\bf e}^1_3+a({\bf e}_4^1+{\bf e}_3 ^2)+{\bf e}^2_4)dz$, 
where $r=\dot{{\mathtt m}}_{22}$ and 
$a=r^{-1}(\dot{{\mathtt m}}_{12}(0)\dot{{\mathtt m}}_{22}-\dot{{\mathtt m}}_{22}(0)\dot{{\mathtt m}}_{12})$. Note that $a$ has a zero of order $1+k_2$ at $0$. 
We assume $(1\pm ia)(z)\neq 0$, $\forall z$. Let $x_{j}:{\rm D}_{\epsilon}\to \C$, $j=0,\dots, 4$, be defined by
\[
\begin{split}
   x_0&=-i\log(\frac{1+ia}{\sqrt{1+a}}),\quad x_1=-\frac{z^{k_1}r(1+a^2)^2}{a'}, \\
    x_j &=\frac{1}{2(r^2(1+a^2)z^{k_1})^{j-1}}\sum_{h=0}^{j-1}\frac{x_{jh}}{z^h},\quad j=2,3,4,
   \end{split}
   \]
where
\[\begin{cases}
x_{20}=\frac{1}{2}(\frac{d}{dz}(\log(\frac{r}{a'}))+\frac{4aa'}{1+a^2}),\\
 x_{21}=\frac{k_1}{2},\\
 x_{30}=-\frac{r^2}{5}(\frac{a'''}{a'}-\frac{r''}{r})+\frac{7r^2}{20}(\frac{d}{dz}(\log(a'))^2-\frac{r^2}{4}(\frac{d}{dz}(\log(r))^2
 -\frac{1}{10}\frac{rr'a''}{a'},\\
x_{31}=-\frac{k_1r^2}{10}\frac{d}{dz}(\log(ra')),\\
 x_{32}=-\frac{k_1r^2(4+k_1)}{20}
 \end{cases}\]
 and
 \[\begin{cases}
\begin{split}x_{40}&=\frac{1}{20}(4r^3\frac{a^\textrm{\sc iv}}{a'}
  -2\frac{a'''}{a'}(11r^2\frac{d}{dz}(\log(a'))+r^2r')+3(2r^2r''-3rr'^2)\frac{d}{dz}(\log(a') \\
&\qquad +3r^2r'(\frac{d}{dz}(\log(a'))^2+21r^3(\frac{d}{dz}(\log(a'))^3-4r^2r''' +3r'(6rr'' -5(r')^2)),
\end{split}
\\
x_{41}=-\frac{k_1r}{20a'^2}(2r^2a'a'''+3a'(2rr'-r^2)a''+3a'^2(3r'^2-2rr''))),\\
x_{42}=-\frac{3}{20}k_1(2+k_1)\frac{d}{dz}(\log(ra'))r^3,\\
x_{43}=-\frac{1}{20}k_1(8+6k_1+k_1^2)r^3.
\end{cases}\]
Consider the maps $X,Y:{\rm D}_{\epsilon}\to \Sp(2,\C)$, defined by
 \[\begin{cases}
 X=\cos(x_0){I}_{4}-\sin(x_0)({\bf e}_2^1 -{\bf e}_1 ^2+{\bf e}_4 ^3-{\bf e}_3 ^4),\\
 \begin{split} \,Y&={\bf e}_2 ^2+{\bf e}_4^4+x_1^{-1}{\bf e}_3 ^3-x_2{\bf e}_4^3+x_3{\bf e}_2^3+
   x_1({\bf e}_1^1+x_2{\bf e}_1^2+(x_4+x_2x_3){\bf e}_1^3+x_3{\bf e}_1^4).\end{split}
    \end{cases}\]
The lift $BXY$ satisfies the required properties.
\end{proof}
 %\vskip0.1cm
 
 %\noindent {\it Lemma 2}. {\it  
 \begin{lemma}\label{l:lemma2-thmB}
 Let $A$ and $\widetilde{A}$ be two lifts of $f$ satisfying \eqref{PE4}. Then  $A^{-1}\widetilde{A}$ is ${\rm H}_1$-valued. 
 Conversely, if $A$ satisfies \eqref{PE4} and if $X$ is ${\rm H}_1$-valued, then $\widetilde{A}=AX$ satisfies \eqref{PE4}. 
\end{lemma}
%\vskip0.1cm

%\noindent{\it Proof of Lemma 2}\,  
\begin{proof}[Proof of Lemma \ref{l:lemma2-thmB}]
Let $X=A^{-1}\widetilde{A}$,
 $\alpha$ and $\widetilde{\alpha}$ be the pull-backs by $A$  and $\widetilde{A}$ of the Maurer--Cartan form. Then,
 \begin{equation}\label{Gaugetran}
     \widetilde{\alpha}=X^{-1}(\alpha X+dX).
       \end{equation}
From $\alpha^3_1=\alpha^4_1=0$ and \eqref{Gaugetran}, we get $\widetilde{\alpha}^3_1=(X^2_1)^2\alpha^4_2$ and  
$\widetilde{\alpha}^4_1=X^2_1X^2_2\alpha^4_2$. 
Since $\widetilde{\alpha}^3_1=\widetilde{\alpha}^4_1=0$ and $\alpha^4_2\neq 0$, we have $X^2_1=0$. 
From $\alpha^3_1=\alpha^4_1=\alpha^2_1-\alpha^4_2=0$, $X^2_1=0$ and \eqref{Gaugetran}, we have
 $\widetilde{\alpha}^2_1-\widetilde{\alpha}^4_2=(X^2_2)^{-1}(X^1_1-(X^2_2)^3)\alpha^4_2$. 
 Since  
 $\widetilde{\alpha}^2_1-\widetilde{\alpha}^4_2=0$, then $X_1^1=(X_2^2)^3$. 
 From $\alpha^3_1=\alpha^4_1=\alpha^2_1-\alpha^4_2=\alpha^1_1-3\alpha^2_2=0$, $X^2_1=X_1^1-(X_2^2)^3=0$ 
 and \eqref{Gaugetran}, we obtain 
 $\widetilde{\alpha}^1_1-3\widetilde{\alpha}^2_2=(3X^2_4X^2_2-4(X^2_2)^{-1}X^1_2)\alpha^4_2$. 
 Hence $X^1_2=\frac{3}{4}(X^2_2)^2X^2_4$. 
 From $\alpha^3_1=\alpha^4_1=\alpha^2_1-\alpha^4_2=\alpha^1_1-3\alpha^2_2=4\alpha^1_2-3\alpha^2_4=0$, $X^2_1=X_1^1-(X_2^2)^3=4X^1_2-3(X^2_2)^2X^2_4=0$ and \eqref{Gaugetran}, we have $4\widetilde{\alpha}^1_2-3\widetilde{\alpha}^2_4=-5(X^2_3X^2_2+X^1_4(X^2_2)^{-1})\alpha^4_2$. 
 Thus, $X^1_4+ X^2_3(X^2_2)^2=0$. 
 From $\alpha^3_1=\alpha^4_1=\alpha^2_1-\alpha^4_2=\alpha^1_1-3\alpha^2_2=\alpha ^1_4=0$, 
 $X^2_1=X_1^1-(X_2^2)^3=X^1_4+ X^2_3(X^2_2)^2=0$ and \eqref{Gaugetran}, we obtain 
 $\widetilde{\alpha}^1_4=((X^2_2)^{-1}X^1_3-\frac{1}{4}X^2_4X^2_3)\alpha^4_2$. Hence, 
 \[
\begin{split}
   X^2_1=X^3_1=X^4_1=X^3_2=X^4_2 &=X_1^1-(X_2^2)^3=X^1_4+ X^2_3(X^2_2)^2\\
   &=4X^1_3-X^2_2X^2_4X^2_3=
   4X^1_2-3(X^2_2)^2X^2_4=0.
  \end{split}
     \]
 This proves that $A$ is $\mathrm{H}_1$-valued.
 Retracing the calculations, one sees that if $A$ satisfies \eqref{PE4} and if $X$ is  $\mathrm{H}_1$-valued,  
 also $AX$ satisfies \eqref{PE4}. 
 \end{proof} 
  %\vskip0.1cm
  
Lemmas \ref{l:lemma1-thmB} and \ref{l:lemma2-thmB} imply that, for every $p_0\in S_{\star}$, there exist an open 
neighborhood $U$ and a cross section $U\to \mathcal F_0$ satisfying \eqref{PE4} and that the transition function of two 
such sections is ${\rm H}_1$-valued. This proves the existence and uniqueness of the reduced bundle ${\mathcal F}$. 
In addition, if $p_0\in |\Delta_f|\cup |\Delta_{f^{\sharp}}|$, then there exist an open neighborhood $U$, such that $U\cap |\Delta_f|\cup |\Delta_{f^{\sharp}}|=\{p_0\}$, and a cross section $A:U\setminus \{p_0\}\to {\mathcal F}$. The point $p_0$ is either a 
removable singularity or a pole. We call $A$ a {\it meromorphic section} of ${\mathcal F}$ at $p_0$. The 
  pull-back of the Maurer--Cartan form by a meromorphic section is holomorphic on $U\setminus \{p_0\}$ and meromorphic on $U$.
 \vskip0.1cm
  
Let $A$, $\hat A$ be two sections of ${\mathcal F}$ such that $\hat A=AX$.  
Put $X={\mathtt S}(x)$, where $x:U\cap \hat U\to \mathrm{SL}(2,\C)_1$ and denote by $\alpha$ and $\hat \alpha$ the 
pull-backs of the Maurer--Cartan form. Then, $\hat \alpha^4_2=(x^1_1)^2\alpha^4_2$ and $\hat \alpha^1_3=(x^1_1)^{-6}\alpha^1_3$.  
This implies that $\varphi^1_3(\varphi^4_2)^3$ is projectable, i.e., there exists a holomorphic quartic differential 
$\delta$ on $S_{\star}$, such that $\pi^*(\delta)=\varphi^1_3(\varphi^4_2)^3$. 
If $p_0\in |\Delta_f|\cup |\Delta_{f^{\sharp}}|$ and $A:U\to {\mathcal F}$ is a meromorphic section defined on an open 
neighborhood of $p_0$, then $\delta|_U=A^{*}(\varphi^1_3(\varphi^4_2)^3)$. This implies that $\delta$ is meromorphic on $U$.
 \vskip0.1cm
 
Let $A$, $\hat A$ be as above. From
$\hat \alpha = {\mathtt S}(x)^{-1}\alpha {\mathtt S}(x)+{\mathtt S}(x)^{-1}d{\mathtt S}(x)$, we have
$\hat A^*(\eta)=x^{-1}A^*(\eta)x+x^{-1}dx$. Taking into account that $A^*(\eta^2_1)$ is nowhere zero, it follows that $\eta$ 
is a Cartan connection.  
\end{proof}

\begin{defn}
We call $(\mathcal F, \eta)$ the projective structure of $f$. A point $p_0\in S$ is said to be \textit{heptactic} if $\delta|_{p_0}=0$. 
An isotropic curve with $\delta=0$ is called a \textit{conformal cycle}.\footnote{We implicitly assume $f(S)$ is not 
properly contained in any other cycle.} If $\delta\neq 0$, $f$  is said to be of {\it general type}.  The {\it conformal bending}  
of an isotropic curve of general type is the meromorphic function  ${\kappa}:= {\mathfrak d}(\delta)^2/\delta$.
\end{defn}

\begin{remark}
In his analysis of isotropic curves in $\C^3$, E. Cartan \cite{Ca4,JMN} defined, for a generic isotropic curve $f:S\to \C^3$ 
({\it minimal curves} in the classical terminology), a nowhere zero holomorphic 1-form $\omega$ on $S$, 
the {\it element of pseudoarc}, and a holomorphic function ${\mathtt k}$, the {\it curvature}. 
One can write $\delta$ and ${\mathfrak d}(\delta)$ in terms of $\omega$ and ${\mathtt k}$. As a result, we obtain $\delta=-\frac{1}{25}(5{\mathtt k}''-4{\mathtt k}^2)\omega^4$, where the derivatives are computed with respect to $\omega$.  If $z:S\to \C$ is a uniformizing parameter for $\omega$, then $f$ is a conformal cycle if and only if either ${\mathtt k}=0$, or else
${\mathtt k}=\sqrt[3]{15/2}\wp_{(0,g_3)}(\sqrt[3]{2/15}z)$, where $\wp_{g_2,g_3}$ is the Weierstrass elliptic $\wp$-function 
with invariants $g_2$ and $g_3$. If ${\mathtt k}=0$, we get a twisted cubic. Assuming $5{\mathtt k}''-4{\mathtt k}^2\neq 0$, 
we have ${\mathfrak d}(\delta)={\mathtt h}\omega^2$, where
$$\mbox{$
{\mathtt h}=\frac{5{\mathtt k}^{(4)}}{2(5{\mathtt k}''-4{\mathtt k}^2)}-\frac{1125{\mathtt k}^{(3)}({\mathtt k}^{(3)}-\frac{16}{5}{\mathtt k}{\mathtt k}')+800{\mathtt k}^{(2)}({\mathtt k}{\mathtt k}^{(2)}+2{\mathtt k}'^2)+512{\mathtt k}^2(\frac{25}{8}{\mathtt k}'^2-{\mathtt k}^3)}{80(5{\mathtt k}''-4{\mathtt k}^2)^2}. $
}
$$
\end{remark}

\subsection{Conformal cycles}  
The group $\SL(2,\C)$ acts on $\CP^3$ and $\Q_3$ via the representation ${\mathtt S}$ (cf. Section 2.1).  
Apart from fixed points, the orbits of this action are 1-dimensional and conformally congruent to each other.  
Let $\widehat{{\mathcal C}}\subset \CP^3$ and ${\mathcal C}\subset \Q_3$ denote, respectively, the orbits through 
$[{\bf e}_1]$ and $[{\bf e}_1\wedge {\bf e}_2]$. 
%Denote by $\widehat{{\mathcal C}}\subset \CP^3$ and ${\mathcal C}\subset \Q_3$ the orbits through 
%$[{\bf e}_1]$ and $[{\bf e}_1\wedge {\bf e}_2]$. 
%The next Proposition solves 
%The equivalence problem 
%for conformal cycles is solved in the following.
%
The equivalence problem for conformal cycles is solved by the following.

\begin{prop}\label{propcycles} 
The orbit ${\mathcal C}$  is a rational conformal cycle and $\widehat{{\mathcal C}}$ is its Legendre associate. In addition,
any other cycle is conformally congruent to ${\mathcal C}$.
 \end{prop}

\begin{proof}
The stabilizer of the action of $\mathrm{SL}(2,\C)$ at $[{\bf e}_1\wedge {\bf e}_2]$ is $\SL(2,\C)_1$. 
Hence ${\mathcal C}$ is biholomorphic to $\CP^1$. The 3-dimensional subgroup $\mathrm{H}={\mathtt S}(\mathrm{SL}(2,\C))$ 
is the maximal integral submanifold through ${I}_{4}$ of the left-invariant completely integrable holomorphic Pfaffian differential 
system  on $\Sp(2,\C)$ given by
\begin{equation}\label{Pfaffian}
\varphi^3_1=\varphi^4_1=\varphi^2_1-\varphi^4_2=\varphi^1_1-3\varphi^2_2=4\varphi^1_2-3\varphi^2_4=\varphi^1_4=\varphi^1_3=0.
   \end{equation}
Then,  $\{(P,A)\in {\mathcal C}\times {\rm H}\mid  [A_1\wedge A_2]=P\}$  is a reduction of the zeroth order frame bundle 
of ${\mathcal C}$ with structure group ${\rm H}_1$. From \eqref{Pfaffian},  it follows that this bundle is the projective bundle of ${\mathcal C}$.  
Since  $\varphi^1_3=0$, the quartic differential of ${\mathcal C}$ vanishes identically. This proves that ${\mathcal C}$ is a a conformal cycle.  The punctured curve ${\mathcal C}_*={\mathcal C}\setminus \{[{\bf e}_3\wedge {\bf e}_4]\}$ is parametrized by $f(z)=[{\mathtt v}(z)\wedge {\mathtt w}(z)]$, where ${\mathtt v}(z)= {\bf e}_1+\frac{z^3}{3}{\bf e}_3-\frac{z^2}{2}{\bf e}_4$ and  
${\mathtt w}(z)= {\bf e}_2-\frac{z^2}{2}{\bf e}_3+z{\bf e}_4$. 
The orbit $\widehat{{\mathcal C}}$ is the twisted cubic $z\mapsto [{\mathtt u}(z)]$, ${\mathtt u}(z)={\mathbf e}_1+z{\mathbf e}_2-\frac{z^3}{6}{\mathbf e}_3+\frac{z^2}{2}{\bf e}_4$. 
Since ${\mathtt w}={\mathtt u}'$ and ${\mathtt v}={\mathtt u}-z{\mathtt w}$,  it follows that $\widehat{{\mathcal C}}$ is 
the Legendre associate of ${\mathcal C}$.

Next, let $f:S\to \Q_3$ be any other conformal cycle and $\mathcal F_f$ be its projective bundle. 
Since $f$ is a cycle, the map $\Psi : (p,A)\in \mathcal F_f\to A\in \Sp(2,\C)$ is an integral manifold of \eqref{Pfaffian}. 
Then, there is $B\in \Sp(2,\C)$ such that $\Psi(\mathcal F_f)\subset  B\cdot {\rm H}$. This implies
 $f(S_{\star})\subset B\cdot {\mathcal C}$. By continuity, $f(S)\subset B\cdot {\mathcal C}$ and, by maximality, 
 $f(S)= B\cdot {\mathcal C}$.\end{proof}
 
 \begin{remark} 
%From Proposition \ref{propcycles} and \cite{Br1988,Br2018} 
According to \cite{Br1988,Br2018}, it follows from Proposition \ref{propcycles}
that conformal cycles 
exhaust the class of isotropic embeddings 
of $\CP^1$ into $\Q_3$
%$\CP^1\to \Q_3$ 
of degree 4.
 \end{remark}
 
 %\begin{figure}[h]
%\begin{center}
%\includegraphics[height=5.6cm,width=11.2cm]{BrFlCycle.pdf}
%\caption{\small{The Bryant surface (left) and the flat front  (right) in ${\mathcal H}^3$ tamed by the cycle ${\mathcal C}$}}\label{FIGURA0}
%\end{center}
%\end{figure}

\begin{defn} 
%The associated surfaces of 
The surfaces associated to a conformal cycle can be viewed as the counterparts of the {\it Cyclides of Dupin} in the 
classical {\it Lie sphere geometry} (see for instance \cite{JMN}).  For this reason they are called {\it pseudo-Cyclides}.  
By construction, all pseudo-Cyclides are
Goursat transforms of the ones tamed by the {\it standard cycle} 
$$
  f(z)=[({\mathtt v}(z)]=\Big[({\bf e}_1+\frac{z^3}{3}{\bf e}_3-\frac{z^2}{2}{\bf e}_4)\wedge ( {\bf e}_2-\frac{z^2}{2}{\bf e}_3+z{\bf e}_4)\Big].
    $$
\end{defn}

\begin{figure}[h]
\begin{center}
\includegraphics[height=5.6cm,width=11.2cm]{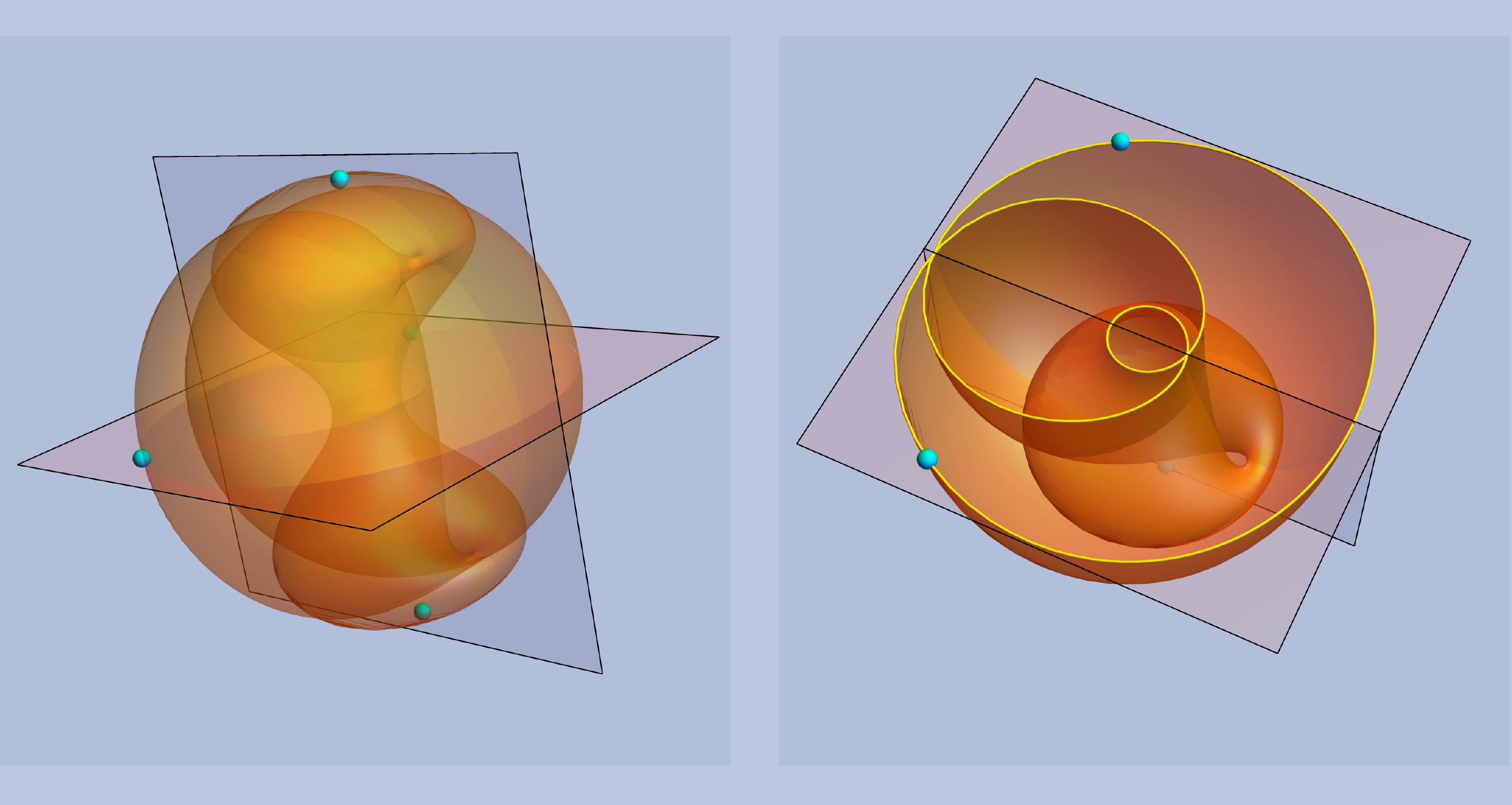}
\caption{\small{A  CMC 1 pseudo-Cyclide of ${\mathcal H}^3$ with four embedded ends.}}\label{FIGURA3}
\end{center}
\end{figure}

\begin{ex}[Standard models for pseudo-Cyclides]\label{ex1}  
%We describe standard models for the pseudo-Cyclides. 
Up to a homothetic factor, the minimal and maximal surfaces tamed by the standard cycle are the {\it Enneper  surface}  
and the {\it conjugate of the maximal space-like Enneper surface of the second kind}  \cite{KO1983}. 
The CMC 1 surface in hyperbolic 3-space tamed by $f$ is a Goursat transform of the rotationally-invariant  
{\it Catenoid cousin} with parameter $\mu = 1$  (cf. \cite{Br1987,JMN}), while the CMC 1 surface in de~Sitter space 
tamed by $f$ is the {\it spacelike Catenoid cousin} considered in \cite{ Lee2005}. The flat front tamed by $f$ is a Goursat 
transform of the rotationally invariant front  with parameter $\mu = -2$ considered in \cite{KUY-O}. Finally, the 
superminimal surface in $S^4$ tamed by $f$ is a Goursat transform of the Veronese embedding  \cite{Ch-S4}. 
Figure \ref{FIGURA3} depicts a CMC 1 pseudo-Cyclide in hyperbolic 3-space. It is not embedded and possesses four smooth ends.
\end{ex}

\subsection{Geometric meaning of heptactic points} 
The 7-dimensional complex homogeneous space 
${\mathfrak C}=\Sp(2,\C)/ {\rm H}$
is the {\it manifold of conformal cycles} of $\Q_3$. 
Let $f$ be an isotropic curve, 
%$p\in S_{\rm I}$ 
$p\in S_\star$ 
a generic point, and 
%$A\in \widehat{\mathcal F}(f)|_p$ 
$A\in {\mathcal F}|_p$
%be a \textcolor{red}{fifth order frame} 
(cf. Theorem \ref{thm:FourthReduction}). The cycle ${\mathcal C}_f|_p= A\cdot {\rm H}\cdot f(p)$ 
is independent of the choice of $A$. We call ${\mathcal C}_f|_p$ the {\it osculating cycle} of $f$ at $p$.
The holomorphic map ${\mathcal C}_f:  {\rm S}_{\star} \ni p\mapsto {\mathcal C}_p\in {\mathfrak C}$
is called 
the {\it osculating map} of $f$.  
Using the projective structure and the Pl\"ucker map, 
%with some lengthy but straightforward computations, 
one can prove the following.

\begin{prop}\label{p:contact}
Let $f$ be an isotropic curve. Then:
\begin{itemize}
\item $f$ and $ {{\mathcal C}_f}|_p$ have analytic contact of order $\ge 5$ at $p$, for every $p\in {\rm S}_{\star}$.
\item $f$ and ${\mathcal C}_f|_p$ have analytic contact of order $>5$ at $p\in S_{\star}$ if and only if $p$ is an heptactic point.
\item The heptactic points of $f$ are the critical points of the osculating map ${\mathcal C}_f$.
\end{itemize}
\end{prop}

%Let $f$ be an isotropic curve, $p\in S_{\rm I}$ a genric point and $A\in {\mathcal F}|_p$. The cycle ${\mathcal C}_f|_p= A\cdot {\rm H}\cdot f(p)$ doesn't depend on the choice of $A$. We call ${\mathcal C}_f|_p$ the {\it osculating cycle} of $f$ at $p$. The holomorphic map ${\mathcal C}_f:p\in  {\rm S}_{{\rm I}}\to {\mathcal C}_p\in {\mathfrak C}$ is the {\it osculating map} of $f$.  
%Using the projective structure  and the Pl\"ucker map one can prove
%
%
%\begin{prop} Let $f$ be an isotropic curve, then
%\begin{itemize}
%\item $f$ and $ {\mathcal C}_f|_p$ have analytic contact of order $\ge 5$ at $p$, for every $p\in {\rm S}_{{\rm I}}$.
%\item $f$ and ${\mathcal C}_f|_p$ have analytic contact of order $>5$ at $p\in S_{{\rm I}}$ if and only if $p$ is an heptactic point.
%\item the heptactic points of $f$ are the critical points of ${\mathcal C}_f$.
%\end{itemize}
%\end{prop}

\begin{remark} 
Proposition \ref{p:contact} implies that, if $\phi$ is tamed by $f:S\to \Q_3$, then the projection of ${\mathcal C}_f|_p$ onto the appropriate 
Riemannian or Lorentzian spaceform  is a pseudo-Cyclide with analytic contact of order at least 5 with $\phi$. 
In addition, $\phi$ is the envelope of the 2-parameter family of its osculating pseudo-Cyclides. This is reminiscent of a 
similar property for surfaces in Lie sphere geometry (cf. \cite{Blaschke}).  Figure \ref{FIGURA3Bis} reproduces a CMC 1 surface tamed 
by a generic isotropic curve (in green) and one of its osculating pseudo-Cyclide (the same one depicted in Figure \ref{FIGURA3}). 
\end{remark}

\begin{figure}[h]
\begin{center}
\includegraphics[height=5.6cm,width=11.2cm]{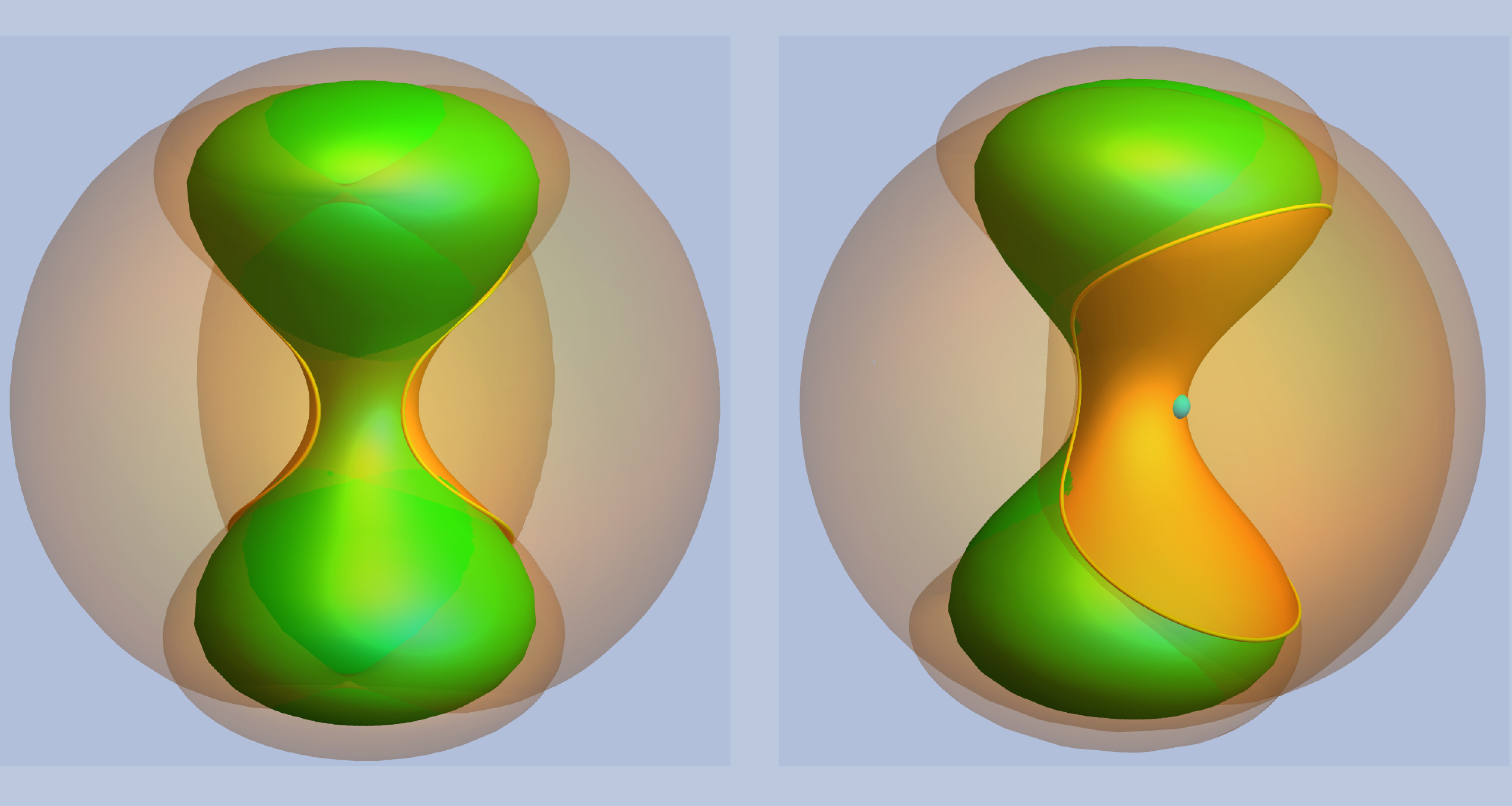}
\caption{\small{A  CMC 1 embedded surface tamed by an isotropic curve of general type (green) and one of its osculating pseudo-Cyclides (orange). The pseudo-Cyclide is the same one of Figure \ref{FIGURA3}.}}\label{FIGURA3Bis}
\end{center}
\end{figure}

\subsection{The equivalence problem}\label{ss:equiv-pbm}  

Let $f:S\to \Q_3$ and $\hat{f}:\hat{S}\to \Q_3$ be two isotropic curves. We say that $\hat{f}$ is dominated by $f$ if 
there exist a holomorphic map $h: \hat{S}\to S$ and $X\in \Sp(2,\C)$ such that $\hat{f}=X\cdot f \circ h$. 

\begin{thmx}\label{thm:equivalence} 
Let $f:S\to \Q_3$ and $\hat f:\hat{S}\to \Q_3$ be two isotropic curves of general type. Then $\hat{f}$ is dominated by $f$ 
if and only f there exists a holomorphic map $h:{\hat S}\to S$ such that $h^*(\delta)=\hat{\delta}$ and 
$h^*({\mathfrak d} \delta)={\mathfrak d}(\hat \delta)$.
\end{thmx}

\begin{proof}   
It suffices to show that two isotropic curves $f$, $\hat f :S\to \Q_3$ of general type
 have the same quartic and quadratic  differentials if and only if $\hat f=X\cdot f$, for some $X\in \Sp(2,\C)$.  
 In the proof we consider ${\mathcal F}_f$ and ${\mathcal F}_{\hat f}$ as projective bundles, with structure group 
 $\mathrm{SL}(2,\C)_1$ acting on the right via the representation ${\mathtt S}$. 
 The proof is 
 %subdivided into 
 organized in
 a first step, 
 two technical lemmas and a second step.
 \vskip0.1cm

\noindent {\bf Step I}.\,  Suppose that
 $\hat f=X\,f$. Then, ${\mathcal F}_{\hat f}=X {\mathcal F}_{f}$ and ${\mathtt F}: (p,A)\in  {\mathcal F}_{f}\to (p,X\, A)\in {\mathcal F}_{\hat f}$ 
 is a bundle map such that ${\mathtt F}^*(\hat \varphi)=\varphi$. Hence, ${\mathtt F}^*(\hat \eta)=\eta$  and $\delta_{f}=\delta_{\hat f}$. 
 Since ${\mathtt F}$ preserves the Cartan connections, the projective atlases induced on $S$ by $({\mathcal F}_f,\eta)$ 
 and $({\mathcal F}_{\hat f},\hat \eta)$ do coincide. This implies that ${\mathfrak d}(\delta)={\mathfrak d}(\hat{\delta})$. 
 
 \begin{lemma} \label{tchLemma}
Let $(P,\eta)$ be a singular projective structure on $S$,  let $(U,z)$ be any complex chart of $S$ 
and let ${\mathfrak p}:U\to P$ be a cross section. 
Put $\epsilon^i_jdz={\mathfrak p}^*(\eta^i_j)$. Then
\begin{equation}\label{first}
{\mathfrak d} |_U= {\mathfrak d}_ {(U,z)} + {\mathfrak r}_{(U,z,{\mathfrak p})}dz^2,
   \end{equation}
where
\begin{equation}\label{addterm}
  {\mathfrak r}_{(U,z,{\mathfrak p})}=
   -2\frac{(\epsilon^2_1)''}{\epsilon^2_1}+\frac{(\epsilon^2_1) '}{\epsilon^2_1}\left( 3\frac{(\epsilon^2_1) '}  {\epsilon^2_1} 
   -4\epsilon^1_1\right) +4\left((\epsilon^1_1)'+(\epsilon^1_1)^2+\epsilon^1_2\epsilon^2_1 \right)
     \end{equation}
     and ${\mathfrak d}_ {(U,z)} $ is defined as in \eqref{secondorderoperator}, although $(U,z)$ is not a projective chart.
  \end{lemma}

\begin{proof} 
It is easy to see that ${\mathfrak r}_{(U,z,{\mathfrak p})}dz^2$ does not depend on the choice of the section ${\mathfrak p}$.  
 Let ${\tilde z}$ be another coordinate and $h$ the transition function ${\tilde z}\circ z^{-1}$. A straightforward computation shows that 
% ${\mathfrak r}_{(U,z,{\mathfrak p})}dz^2={\mathfrak r}_{(U,{\tilde z},{\mathfrak p})}d{\tilde z}^2-2z^*({\mathcal S}(h))$.
${\mathfrak r}_{(U,z,{\mathfrak p})}dz^2={\mathfrak r}_{(U,{\tilde z},{\mathfrak p})}d{\tilde z}^2-2 \mathcal{S}_z(h) \,dz^2$.
 From \eqref{ftrans} it follows that the right hand side of \eqref{first} does not depend on the choice of the local coordinate 
 and of the cross section.
 If $(U,z)$ is a projective chart and ${\mathfrak p}$ is a flat section, then ${\mathfrak r}_{(U,z,{\mathfrak p})}=0$. This implies the result.
 \end{proof}
 
 \begin{lemma}\label{projreduction}
 Let $(P,\eta)$ be a singular projective structure on $S$ with singular locus $D$. Let $\delta$ be a meromorphic quartic differential and $S_{\sharp}=S\setminus D\cup |\Delta_{\delta}|$. 
 There exists a unique reduced bundle 
 %$P|_{S_{\sharp}}$ admits a unique reduced bundle 
$\widehat{\pi}:{\widehat P}\to S_{\sharp}$ of $P|_{S_{\sharp}}$ with structure group 
${\rm Z}_8=\{\epsilon {\bf a}^1_1+\epsilon^{-1}{\bf a}^2_2\mid \epsilon^8=1\}$, such that
$\widehat{\pi}^*(\delta)=6^{3/4}(\eta^2_1)^4$ and $\eta^1_1=0$. In addition, $\widehat{\pi}^*({\mathfrak d}(\delta))=4\eta^2_1\eta^1_2$.
\end{lemma}

\begin{proof} 
Let ${\mathtt D}: P|_{S_{\sharp}}\to \C$ be the holomorphic function such that 
%$\widehat{\pi}^*(\delta)={\mathtt D}(\eta^2_1)^4$. 
{${\pi}^*(\delta)={\mathtt D}(\eta^2_1)^4$}.
By construction ${\mathtt D}(\mathfrak p\cdot x)=(x_1^1)^8 {\mathtt D}(\mathfrak p)$, for every ${\mathfrak p}\in P|_{S_{\sharp}}$ 
and every $x\in {\SL}(2,\C)_1$. Thus, $\hat P=\{{\mathfrak p}\in P|_{S_{\sharp}}\mid {\mathtt D}({\mathfrak p})= 6^{3/4} \}$
is a reduction of $P|_{S_{\sharp}}$ with structure group ${\rm H}=\{x\in {\SL}(2,\C)_1 \mid (x^1_1) ^8=1\}$ 
such that $\hat \pi\, ^*(\delta)=6^{3/4}(\eta^2_1)^4$. 
The 1-form $\eta^1_1$ is tensorial on $\hat P$. Thus, $\eta^1_1={\mathtt E}\eta^2_1$, where ${\mathtt E}$ is a holomorphic function such that ${\mathtt E}({\mathfrak p}x)=\epsilon^{-2}({\mathtt E}({\mathfrak p})-\epsilon x^1_2)$.
Then, ${\widehat P}=\{{\mathfrak p}\in \hat P\mid {\mathtt E}({\mathfrak p})=0\}$ is the unique reduction of $P|_{S_{\sharp}}$ such that 
$\widehat{\pi}^*(\delta)=6^{3/4}(\eta^2_1)^4$, $\eta^1_1=0$. 
Let $(U,z)$ be a complex chart  and ${\mathfrak p}:U\to \hat P$ be a cross section such 
that $\delta =6^{3/4} dz^4$. Then $\epsilon ^1_1 =0$ and $(\epsilon^2_1)^4 =1$. Using the Lemma \ref{tchLemma},
we have ${\mathfrak r}_{(U,z,{\mathfrak p})} = 4\epsilon^1_2\epsilon^2_1$. Therefore, 
${\mathfrak d}(\delta) =4{\mathfrak p}^*(\eta^2_1\eta^1_2))$.
Taking into account that $\eta^2_1$ and $\eta^1_2$ are tensorial on $\hat P$, we have $\hat{\pi}^*({\mathfrak d}(\delta))=4\eta^2_1\eta^1_2$.
%
%To prove that $\widehat{\pi}^*({\mathfrak d}(\delta))=4\eta^2_1\eta^1_2$ we use Lemma \ref{tchLemma}.  
%Let $(U,z)$ be a complex chart with $U$ simply connected and  ${\mathfrak p}:U\to \hat P$ be a cross section such 
%that $\delta =6^{3/4} dz^4$. Then, in view of the Lemma, 
%%
%we have ${\mathfrak d}(\delta)= {\mathfrak r}_{(U,z,{\mathfrak p})}$. 
%Consider the map $x=x^1_1{\bf a}^1_1+x_1^1\, ^{-1}{\bf a}^2_2+x^1_2{\bf a}^2_1$ defined by $x^1_1=\sqrt{\epsilon^2_1}$ and 
% $x^1_2=\epsilon^2_1\,^{3/2}(\epsilon^2_1\,' -2\epsilon^1_1\epsilon ^2_1)$. 
% Then, $\hat {\mathfrak p}={\mathfrak p}x$ is a cross section of $\hat P$. Computing \textcolor{red}{$\hat {\mathfrak p}^*(\eta)=x^{-1}{\mathfrak p}^*(\eta)x+x^{-1}dx$} we get
%$\hat {\mathfrak p}^*(\eta^2_1)=dz$ and   $\hat {\mathfrak p}^*(\eta^1_2)=\frac{1}{4}{\mathfrak r}_{(U,z,{\mathfrak p})}dz$. 
%This implies that for every $p_0\in S_{\sharp}$ there exists a cross section $\hat {\mathfrak p}$ of $\hat P$, defined on an open 
%neighborhood of $p_0$, such that  $4\hat{\mathfrak p}^*(\eta^2_1\eta^1_2))={\mathfrak d}(\delta)$.  
%Taking into account that $\eta^2_1\eta^1_2$ is tensorial on $\hat P$, we have $\hat{\pi}^*({\mathfrak d}(\delta))=4\eta^2_1\eta^1_2$.
 \end{proof}
 
 \noindent  {\bf Step II}.\, Next, assume $\delta=\hat \delta$ and  ${\mathfrak d}(\delta)={\mathfrak d}(\hat \delta)$. 
 Let $D_{\star}=|\Delta_{f}|\cup | \Delta_{f^{\sharp}}|\cup |\Delta_{\hat f}|\cup | \Delta_{{\hat f}^{\sharp}}|\cup |\Delta_{\delta}|$ 
 and  ${\mathtt p}:\widetilde{S}\to S\setminus D_{\star}$ be the universal covering. 
 Since ${\mathtt p}^{*}(\delta)$ is a nowhere zero holomorphic differential, $\widetilde{S}$ is biholomorphic either to $\C$, or to the unit disk. Then, there exists $z:\widetilde{S}\to \C$ such that $\delta = dz^4$.  
 Consider the ${\mathrm Z}_8$-reductions $\hat{\mathcal F}_{f\circ {\mathtt p}}$ and 
 $\hat{\mathcal F}_{\hat f\circ {\mathtt p}}$ of the projective bundles  ${\mathcal F}_{f\circ {\mathtt p}}$ and  ${\mathcal F}_{\hat f\circ {\mathtt p}}$ constructed in Proposition \ref{projreduction}. Since $\widetilde{S}$ is simply connected, they are trivial.
 Pick two global trivialisations $A$ and $\hat A$ and denote by $\alpha$, $\hat \alpha$ the pull-backs of the Maurer--Cartan form.
 Possibly acting on the right with an element of ${\mathrm Z}_8$ we have $\alpha^4_2=\hat{\alpha}^4_2=dz$. Since $\alpha^4_2\alpha^2_4=\hat{\alpha}^4_2\hat{\alpha}^2_4={\mathfrak d}(\delta)$, then $\alpha=\hat{\alpha}$. 
 By the Cartan--Darboux congruence Theorem \cite{JMN}, there exists $X\in \Sp(2,\C)$ such that $\hat A=X\,A$. 
 This implies $\hat f\circ {\mathtt p} = X f\circ {\mathtt p}$. Hence $\hat f=X\, f$, as claimed.
 \end{proof}

\begin{prop}\label{existence}
Let $S$ be a simply connected Riemann surface. Let $\delta$ and $\gamma$ be 
two holomorphic differentials of degree four and two, respectively. If $\delta_p \neq 0$,
for every $p\in S$, then there exists an isotropic curve $f : S \to \Q_3$ of general type, such that
$ \delta_f = \delta$ and  ${\mathfrak d}(\delta_f) = \gamma$.
Moreover, $f$ is unique up to the action by an element of the symplectic group.
\end{prop}

\begin{proof}
By the Uniformization Theorem, taking into account that on the Riemann surface $S$ there 
is a non-null holomorphic differential $\delta$, it follows that $S$ is equivalent to either the complex plane or the unit disk.
On $S$, we can then consider a global holomorphic coordinate $z : S \to \C$ and write
$\gamma = \Gamma dz^2$, $\delta = D\,dz^4$ with $D_{|p} \neq 0$, for each $p\in S$. Since $S$ is simply connected, we can choose a
fourth root of $D$, say $\sqrt[4]{D}$. Next, let $E : = \Gamma/\sqrt[4]{D}$
  and consider the  holomorphic $\mathfrak{sp}(2,\C)$-valued 1-form 
  $\hat \varphi =({\bf e}^1_2+{\bf e}^2_4+{\bf e}^3_1-{\bf e}^4_3)\sqrt[4]{D}dz+({\bf e}^4_2+\frac34( {\bf e}^2_1-{\bf e}^3_4))Edx$. 
  Now, $\hat \varphi $ satisfies the Maurer--Cartan equation
 and hence, by the Cartan--Darboux existence theorem, there exists a holomorphic map 
 $A : S \to \Sp(2,\C) $ such that $ \hat \varphi = A^{-1} d A$.
  The map $ f : S\ni p \longmapsto [A_1(p) \wedge A_2(p)] \in \Q_3$ defines an isotropic curve with the required proprieties. The uniqueness assertion follows from Theorem \ref{thm:equivalence}. \end{proof}
  
Theorem \ref{thm:equivalence} and Proposition \ref{existence} can be rephrased in terms of projective structures.

\begin{cor}
Let $f:S\to \Q_3$ and $\hat f:\hat{S}\to \Q_3$ be two isotropic curves of general type. 
Then $\hat{f}$ is dominated by $f$ if and only if there exists a projective map $h:{\hat S}\to S$ such that $h^*(\delta)=\hat{\delta}$.
\end{cor}
 
\begin{cor} 
Let $(P,\eta)$ be a projective structure on $S$ with singular locus $D$ and let $\delta$ be a nonzero meromorphic 
quartic differential on $S$. Let ${\mathtt p}: S_{\star}\to S\setminus D\cup|\Delta_{\delta}|$ be a universal covering. 
Then there exists a generic isotropic curve $f:S_{\star}\to \Q_3$ whose projective structure is equivalent 
to $({\mathtt p}^*(P),{\mathtt p}^*(\eta))$ such that ${\mathtt p}^*(\delta)=\delta_f$.
\end{cor}

\subsection{Conformal deformation and rigidity} 

We adapt to our specific context, the general concepts of deformation and rigidity \cite{Ca6,Gr,JensenJDG,JM}.

\begin{defn}  
Let $f$, $\hat f:S\to \Q_3$ be two isotropic curves. We say that $\hat f$ is a {\it $k$th order conformal deformation} of $f$ if there exists a holomorphic map ${\mathtt D}: S\to \Sp(2,\C)$,
 such that $\hat f$ and ${\mathtt D}(p)\cdot f$ have analytic contact of order $k$ at $p$, for every $p\in S$. 
 A  deformation is {\it trivial} if it is congruent to $f$. 
 We say that $f$ is {\it deformable of order $k$} if, for every $p_0\in S\setminus |\Delta_f|\cup |\Delta_{f^{\sharp}}|\cup |\Delta_{\delta}|$, there exist an open neighborhood $U$ of $p_0$ and an isotropic curve of general type $\hat f:U\to \Q_3$, such that $\hat f$ is a non-trivial $k$th order deformation of $f|_U$. Otherwise, $f$ is said to be {\it rigid to order k}.
 \end{defn}
 
 \begin{thmx}\label{thm:def-rig}
 An isotropic curve of general type is deformable {of} order four and is rigid to order five.
 \end{thmx}
 
 \begin{proof} We first recall the following.
 \vskip0.1cm

\noindent {\it Fact}.\, Two holomorphic maps $\psi$, $\hat \psi : S \to \CP^4$  have the same $k$th order jets at $p_0$ 
if and only if for any lifts $\Psi,\hat \Psi:S\to \C^5$ of $\psi$ and $\hat \psi$ and for every complex chart $(U,z)$ with $p_0\in U$, 
there exist
$\varrho_j\in \big(\Omega^{1,0}(S)|_{p_0}\big)^n$, $n=0,\dots, k$, such that 
\begin{equation}\label{contact}
  \delta^j(\hat \Psi)|_{p_0}=\sum_{i=0}^{j}c_i^j \varrho_i\delta^{j-i}(\Psi)|_{p_0},
  \quad j=0,\dots ,k,\quad \delta^k(\Psi)=\frac{d\Psi^k}{dz^k}(dz)^k,    
   \end{equation}
 where $c^j_i\in {\mathbb N}$ are defined by $c^j_0=c^j_j=1$ and by $c^j_i=c^{j-1}_{i-1}+c_i^{j-1}$, for every $j\ge 2$ and $1\le i< j$.
 \vskip0.1cm
 
 \noindent Since the 
 %Theorem 
 result
 is local, we may assume the existence of a global section $A$ of the ${\mathrm Z}_8$-bundle ${\mathcal F}_f$. Recall that 
 \begin{equation}\label{eq1}
   A^{-1}dA=({\bf e}^1_2+{\bf e}^2_4+{\bf e}^3_1-{\bf e}^4_3)\zeta+ \big({\bf e}^4_2+\frac{3}{4}({\bf e}^2_1-{\bf e}^3_4)\big)\eta
      \end{equation}
 where $\zeta,\eta$ are holomorphic 1-forms and $\zeta|_p\neq 0$. 
 Without loss of generality, we may suppose that $\zeta$ is nowhere zero. It defines on $S$ a unimodular affine structure 
 consisting of all complex charts $(U,z)$ such that $\zeta=dz$, $\eta = b dz$. 
 Let $\hat f$ be another isotropic curve of general type and $\hat A$ be a cross section of ${\mathcal F}_{\hat f}$. 
 Then
  \begin{equation}\label{eq2}
   \hat A^{-1}d\hat A=
      \hat a({\bf e}^1_2+{\bf e}^2_4+{\bf e}^3_1-{\bf e}^4_3)dz+ \hat b\big({\bf e}^4_2+\frac{3}{4}({\bf e}^2_1-{\bf e}^3_4)\big)dz,    
        \end{equation}
 where $\hat a$, $\hat b$ are holomorphic functions, $\hat a \neq 0$. Without loss of generality, we suppose that 
 $\hat a$ is nowhere zero. The quadratic differential
 $$
   {\mathfrak s}= \Big(2\frac{\hat a ''}{\hat a} - 3\Big(\frac{d\log \hat a}{dz}\Big)^2\Big)dz^2
 %   {\mathfrak s}= \Big(2\frac{\hat a ''}{\hat a} - 3\big(\frac{d}{dz}(\log \hat{a})\big)^2\Big)dz^2
$$
 does not depend on the choice of the unimodular affine chart.

 \begin{lemma}\label{l:4th-order}
 $\hat f$ is a fourth order conformal deformation of $f$ if and only if $\gamma_{\hat f}=\gamma_f+\mathfrak s$.
   \end{lemma}
 
\begin{proof} 
Let $\lambda :\Q_3\to {\mathcal Q}_3$ be the Pl\"ucker map and 
${\mathtt L}:\Sp(2,\C)\to  {\rm O}({\mathfrak C}^5,g_{\mathfrak C})$ be the spin covering homomorphism. Consider
$\psi =\lambda\circ f$, $\hat \psi=\lambda\circ \hat f$ and let
${\mathcal A},  \hat{\mathcal A}:S\to {\rm O}({\mathfrak C}^5,g_{\mathfrak C})$ be the maps defined by
${\mathcal A}={\mathtt L}\circ A$, $\hat{\mathcal A}={\mathtt L}\circ \hat A$. Then, ${\mathcal A}_1$ is a 
lift of $\psi$ and $\hat {\mathcal A}_1$ is a lift of $\hat \psi$. In addition, ${\mathcal A}$ and $\hat {\mathcal A}$ 
satisfy ${\mathcal A}^{-1}d\mathcal A={\mathtt N}dz$ and 
$\hat{{\mathcal A}}^{-1}d\hat{\mathcal A}=\hat{{\mathtt N}}dz$, where ${\mathtt N},\hat{{\mathtt N}}:S\to {\mathfrak o}({\mathfrak C}^5,g_{\mathfrak C})$ are given by
\begin{equation}\label{N}
 \begin{cases}
{\mathtt N}=({\bf b}^1_2+\sqrt{2}({\bf b}^3_4-{\bf b}^2_3)-{\bf b}^4_1-{\bf b}^4_5+{\bf b}^5_2) +
b({\bf b}^2_1-{\bf b}^5_4+\frac{3}{3\sqrt{2}}({\bf b}^4_3-{\bf b}^3_2)),\\
  \hat{{\mathtt N}}=\hat a({\bf b}^1_2+\sqrt{2}({\bf b}^3_4-{\bf b}^2_3)-{\bf b}^4_1-{\bf b}^4_5+{\bf b}^5_2) +
   \hat b({\bf b}^2_1-{\bf b}^5_4+\frac{3}{3\sqrt{2}}({\bf b}^4_3-{\bf b}^3_2)).
       \end{cases}
          \end{equation}
Let 
%${\mathtt F}_{(h)}$, $\hat{{\mathtt F}}_{(h)}:S\to {\mathfrak C}^5$ 
${\mathtt F}_{(h)}$, $\hat{{\mathtt F}}_{(h)}:S\to \mathbb C^5$ 
be given by the recursive formulae
\begin{equation}\label{rec}
   {\mathtt F}_{(0)}=\hat{{\mathtt F}}_{(0)}={\bf b}_1,\quad  {\mathtt F}_{(h)}=(\frac{d}{dz}+{\mathtt N}){\mathtt F}_{(h-1)},\quad 
     \hat{{\mathtt F}}_{(h)}=(\frac{d}{dz}+\hat{{\mathtt N}})\hat{{\mathtt F}}_{(h-1)},
    \end{equation}
    where $({\bf b}_1, \dots, {\bf b}_5)$ is the canonical basis of $\mathbb C^5$.
Define ${\mathtt F}$, $\hat{\mathtt F}:S\to {\rm GL}(5,\C)$ by ${\mathtt F}=({\mathtt F}_{(0)},\dots , {\mathtt F}_{(4)})$ and by 
$\hat{{\mathtt F}}=(\hat {\mathtt F}_{(0)},\dots , \hat {\mathtt F}_{(4)})$. Then 
\begin{equation}\label{J}J^{(4)}({\mathcal A}_1)={\mathcal A}\cdot {\mathtt F},\quad  
J^{(4)}(\hat {\mathcal A}_1)=\hat {\mathcal A}\cdot \hat {\mathtt F},
\end{equation}
 where
$$
    \begin{cases}J^{(4)}({\mathcal A}_1)=
     \Big({\mathcal A}_1,\frac{d{\mathcal A}_1}{dz},\frac{d^2{\mathcal A}_1}{dz^2},\frac{d^3{\mathcal A}_1}{dz^3},\frac{d^4{\mathcal A}_1}{dz^4}\Big),\\
J^{(4)}(\hat {\mathcal A}_1)=\Big(\hat{\mathcal A}_1,\frac{d\hat{\mathcal A}_1}{dz},\frac{d^2\hat{\mathcal A}_1}{dz^2},\frac{d^3\hat{\mathcal A}_1}{dz^3},\frac{d^4\hat {\mathcal A}_1}{dz^4}\Big).
     \end{cases}
      $$
From the fact mentioned at the beginning of the proof, $\hat f$ is a fourth order deformation of $f$ if and only if 
there exist a holomorphic map
${\mathcal D}:S\to {\rm O}({\mathfrak C}^5,g_{\mathfrak C})$ and holomorphic functions $r_j$, $j=0,\dots, 4$, such that 
\begin{equation}\label{def}
   J^{(4)}(\hat {\mathcal A}_1)={\mathcal D}\cdot J^{(4)}({\mathcal A}_1)\cdot {\rm R},
       \end{equation}
 where
\begin{equation}\label{R}
   {\rm R}=r_0 {I}_{5}
       +r_1({\bf b}_1^2+2{\bf b}_2^3+3{\bf b}_3^4+4{\bf b}_4^5)+ r_2({\bf b}_1^3 +3{\bf b}_2^4 + 6 {\bf b}_3^5) +   r_3({\bf b}_1^4+4{\bf b}_2^5)+r_4{\bf b}_1^5.
         \end{equation}
By construction, we have 
\begin{equation}\label{Rel}
   {\mathcal D}=\hat{\mathcal A}\,\hat{\mathtt F}\,\mathrm{R}^{-1}\,{\mathtt F}^{-1}\,{\mathcal A}^{-1}.
       \end{equation}
This implies that ${\mathtt F}{\rm R}\hat{\mathtt F}^{-1}$ takes values in the orthogonal group ${\rm O}({\mathfrak C}^5,g_{\mathfrak C})$. Imposing the orthogonality conditions, we obtain, letting $\varepsilon = \pm 1$, 
\begin{equation}\label{r}
\mbox{$\footnotesize
\begin{cases}
r_0=\varepsilon{\hat a}^2,\\
r_1=2\varepsilon\hat a \hat a',\\
r_2=\frac{\varepsilon}{7}(5(\hat a^3\hat b - \hat a^2 b) +29\hat a'^2+4\hat a \hat a''),\\
r_3=\frac{\varepsilon}{42\hat a}\left( 14\hat a\hat a'''+132\hat a\hat a'\hat a'' +(390\hat a'^2-200\hat a^2b+235 \hat a^3\hat b)\hat a' -35\hat a^3b'+35\hat a^4\hat b'\right),\\
%
%\begin{split} 
%r_4 =&\frac{\varepsilon}{294\hat a^2}\left(1372\hat a^2\hat a' \hat a''' + (1266\hat a^4\hat b+480\hat a\hat a'^2-720\hat a^3b+624\hat a^2\hat a'')\hat a''\right)+\\
%&\frac{\varepsilon}{294\hat a^2}\left((2548\hat a^4\hat b'-1960\hat a^3b')\hat a' +(8342\hat a^3\hat b-6760\hat a^2b)\hat a'^2+9195\hat a'^4\right)+\\
%& \frac{\varepsilon}{294\hat a^2}\left(588\hat a^8+(294\hat b''-900 b\hat b)\hat a^5 -(294(2+\hat b'')+513b^2)\hat a^4 +387{\hat a}^6{\hat b}^2 \right),
%\end{split}\\
%
\begin{split} 
r_4 =&\frac{\varepsilon}{294\hat a^2}\left(1372\hat a^2\hat a' \hat a''' + (1266\hat a^4\hat b+480\hat a\hat a'^2-720\hat a^3b+624\hat a^2\hat a'')\hat a''\right.+\\
&\quad \qquad \left.(2548\hat a^4\hat b'-1960\hat a^3b')\hat a' +(8342\hat a^3\hat b-6760\hat a^2b)\hat a'^2+9195\hat a'^4\right.+\\
&\quad \qquad \left.588\hat a^8+(294\hat b''-900 b\hat b)\hat a^5 -(294(2+\hat b'')+513b^2)\hat a^4 +387{\hat a}^6{\hat b}^2 \right),
\end{split}
\end{cases} $}
\end{equation}
and, in addition,
\begin{equation}\label{constitutive}
  \hat b=\frac{1}{\hat a}b+2\frac{\hat a''}{\hat a^2}-3\frac{\hat a'^2}{\hat a^3}.
      \end{equation}
This proves that  $\gamma_{\hat f}=\gamma_f+\mathfrak s$. 
Conversely, if  $\gamma_{\hat f}=\gamma_f+\mathfrak s$, then $\hat b$ is as in \eqref{constitutive}. 
Define $r_0,\dots, r_4$, ${\rm R}$, and ${\mathcal D}$ as in \eqref{r}, \eqref{R} and \eqref{Rel}. 
Then ${\mathcal D}$ is $ {\rm O}({\mathfrak C}^5,g_{\mathfrak C})$-valued. From \eqref{Rel} it follows 
that ${\mathcal A}_1$ and $\hat {\mathcal A}_1$ satisfy \eqref{def}. 
This implies that $\hat f$ is a fourth order deformation of $f$.
\end{proof}

As a consequence of the Lemma \ref{l:4th-order}, it follows that an isotropic curve of general type 
is deformable of order four and that its local deformations depend on one arbitrary holomorphic function. 
We conclude the proof by showing that an isotropic curve of general type is rigid to order five. 
Let $\hat f$ be a fifth order deformation of $f$. Since $\hat f$ is a fourth order deformation of $f$, 
then $\hat b$ is as in \eqref{constitutive} and the lifts ${\mathcal A}_1$, $\hat {\mathcal A}_1$ of $\psi$ and $\hat \psi$ 
satisfy \eqref{def}, where ${\rm R}$ and ${\mathcal D}$ are as in \eqref{R},  \eqref{r} and \eqref{Rel}. 
Since $\hat f$ is a fifth order deformation, there exists a holomorphic function $r_5$, such that 
$$
   \frac{d^5\hat{\mathcal A}_1}{dz^5}={\mathcal D}\cdot \left(J^{(4)}({\mathcal A}_1)(r_5{\bf b}_1+5r_4{\bf b}_2+10r_3{\bf b}_3+10r_2{\bf b}_4+5r_1{\bf b}_5)+r_0\frac{d^5{\mathcal A}_1}{dz^5}\right),
          $$
where $({\bf b}_1,\dots, {\bf b}_5)$ is the standard basis of $\C^5$. Taking into account that
$$
 \frac{d^5{\mathcal A}_1}{dz^5}
  ={\mathcal A} {\mathtt F}_{(5)},\quad \frac{d^5\hat{\mathcal A}_1}{dz^5}=\hat {\mathcal A} \hat {\mathtt F}_{(5)},
 $$
we obtain
$$
 \hat{\mathcal A}\hat {\mathtt F}_{(5)}
  ={\mathcal D}{\mathcal A}({\mathtt F}(r_5{\bf b}_1+5r_4{\bf b}_2+10r_3{\bf b}_3+10r_2{\bf b}_4)+r_0{\mathtt F}_{(5)}).
    $$
Using \eqref{Rel}, we have
$$
 \hat {\mathtt F}_{(5)}
    =\hat{\mathtt F}{\rm R}^{-1}((r_5{\bf b}_1+5r_4{\bf b}_2+10r_3{\bf b}_3+10r_2{\bf b}_4)+r_0{\mathtt F}^{-1}{\mathtt F}_{(5)}),
       $$
which, in turn, implies
\begin{equation}\label{syst}
   {\mathtt F}{\rm R}\hat{{\mathtt F}}^{-1}\hat{{\mathtt F}}_{(5)}  = {\mathtt F}(r_5{\bf b}_1
      +5r_4{\bf b}_2+10r_3{\bf b}_3+10r_2{\bf b}_4)+r_0{\mathtt F}_{(5)}.
        \end{equation}
The second scalar equation of the system \eqref{syst} implies $6\hat a(1-\hat a^4)=0$. 
Then $\hat a$ is a fourth root of unity. Possibly acting on the right of $\hat A$ with an element of ${\rm Z}_8$, 
we may assume $\hat a=1$. From  \eqref{constitutive}, we have $\hat b=b$. 
Then, \eqref{eq1} and \eqref{eq2} imply that $\hat A=X\cdot A$, for a unique $X\in \Sp(2,\C)$. 
This proves that $f$ and $\hat f$ are conformally congruent to each other.
  \end{proof}

\section{Isotropic curves with constant bending}\label{s:cc isotropic curves}

\subsection{Isotropic $W$-curves}  

Let $q=m/n\in \Q$ be a nonzero rational number, $|q|\neq 1$. 
Put 
%$\varepsilon_q=(m+n)_{{\rm mod} 2}$ 
$\varepsilon_q = m+n \pmod{2}$ and define ${\mathtt v}_q, {\mathtt w}_q:\C\to \C^4$ by
\begin{equation}\label{w}
\begin{cases}
{\mathtt v}_q (z) =(m-n) {\bf e}_1-(m+n)z^{(1+\varepsilon_q)m} {\bf e}_3-
2i\sqrt{mn}z^{\frac{(m+n)(1+\varepsilon_q)}{2}}{\bf e}_4 ,\\
{\mathtt w}_q(z) = (m-n){\bf e}_2-2i\sqrt{mn}z^{\frac{(m+n)(1+\varepsilon_q)}{2}}{\bf e}_3+(m+n)z^{n(1+\varepsilon_q)}{\bf e}_4 .
\end{cases}
\end{equation}
Let $f_q:\C\cup \{\infty\}\to  \Q_3$ be the one-to-one isotropic curve defined by  $f_q(z)= [{\mathtt v}_q(z)\wedge {\mathtt w}_q(z)]$, $f_q(\infty)=[{\bf e}_3\wedge {\bf e}_4]$. 
%Adhering to the classical terminology, 
We call $f_q$ the {\it isotropic $W$-curve} with parameter
%\footnote{Properly speaking, the classical w-curves are the Legendrian associates.} 
%\footnote{Properly speaking, the classical $W$-curves are the Legendre associates of the $f_q$.} 
$q$.
%\footnote{The Legendre associate of $f_q$ is a $W$-curve in the classical sense \cite{RB,Ch-S4}.} 
The Legendre associate of $f_q$ is computed to be
\begin{equation}\label{aw}
 f^\sharp_q(z)=[i\sqrt{n}({\bf e}_1+iz^{(1+\varepsilon_q)m}{\bf e}_3)+\sqrt{m}(z^{\frac{(1+\varepsilon_q)(m-n)}{2}}{\bf e}_2
-z^{\frac{(1+\varepsilon_q)(m+n)}{2}}{\bf e}_4)],
      \end{equation}
which is a $W$-curve in the classical sense (cf. \cite{RB,Ch-S4}). This motivates the terminology.
If $q=\pm \,3,\pm \,1/3$,  $f_q$ is a conformal cycle. For all other values of $q$, the isotropic $W$-curve is of general type.  
Assuming $q\neq \pm \,3,\pm \,1/3$, one can
explicitly build a global meromorphic section of the ${\mathrm Z}_8$-bundle $\widetilde{\mathcal F}$, holomorphic 
on $\dot{\C}$, with poles at $0$ and $\infty$. Consequently, one can compute the meromorphic differentials to obtain
%. As a result, one gets
$$
    \delta_{q}=- \frac{9m^4-82m^2n^2+9n^4}{100z^4}dz^4,\quad \gamma_{q}=\frac{2(m^2+n^2)}{5z^2}dz^2.
       $$
Thus, $f_q$ has constant bending 
\begin{equation}\label{bendingw}
     \kappa_q=-16(1+q^2)^2(9q^4-82q^2+9)^{-1},\quad q\in \Q\setminus \{\pm \,3, \pm \,1/3\}.
        \end{equation}
In particular, $f_q$, $f_{q^{-1}}$, $f_{-q}$ and $f_{-q^{-1}}$ have the same differentials. 
By Theorem \ref{thm:equivalence},
from the viewpoint of conformal geometry, these curves are then equivalent to each other. Consequently, we may assume  $q>1$. 
From
\cite{Br1988,Br2018}, using \eqref{w} and  \eqref{aw}, it follows that the degree 
$d_{f_q}$ and the ramification degree $r_{f_q}$ of $f_q$ are
 $r_{f_q}=2((1+\varepsilon_q)n-1)$ and  $d_{f_q}=(1+\varepsilon_q)(m+n)$. In particular, if $q=2p+1$ is an odd integer, $f_q$ 
 is 
 %a null 
 an isotropic 
 immersion of degree
{$2(p+1)$}. Referring again to \cite{Br1988,Br2018}, conformal cycles exhaust the class of 
%null 
isotropic immersions of degree 4 and each 
isotropic
%null 
immersion of degree 6 is conformally equivalent to $f_5$.

\begin{figure}[h]
\begin{center}
\includegraphics[height=5.6cm,width=11.2cm]{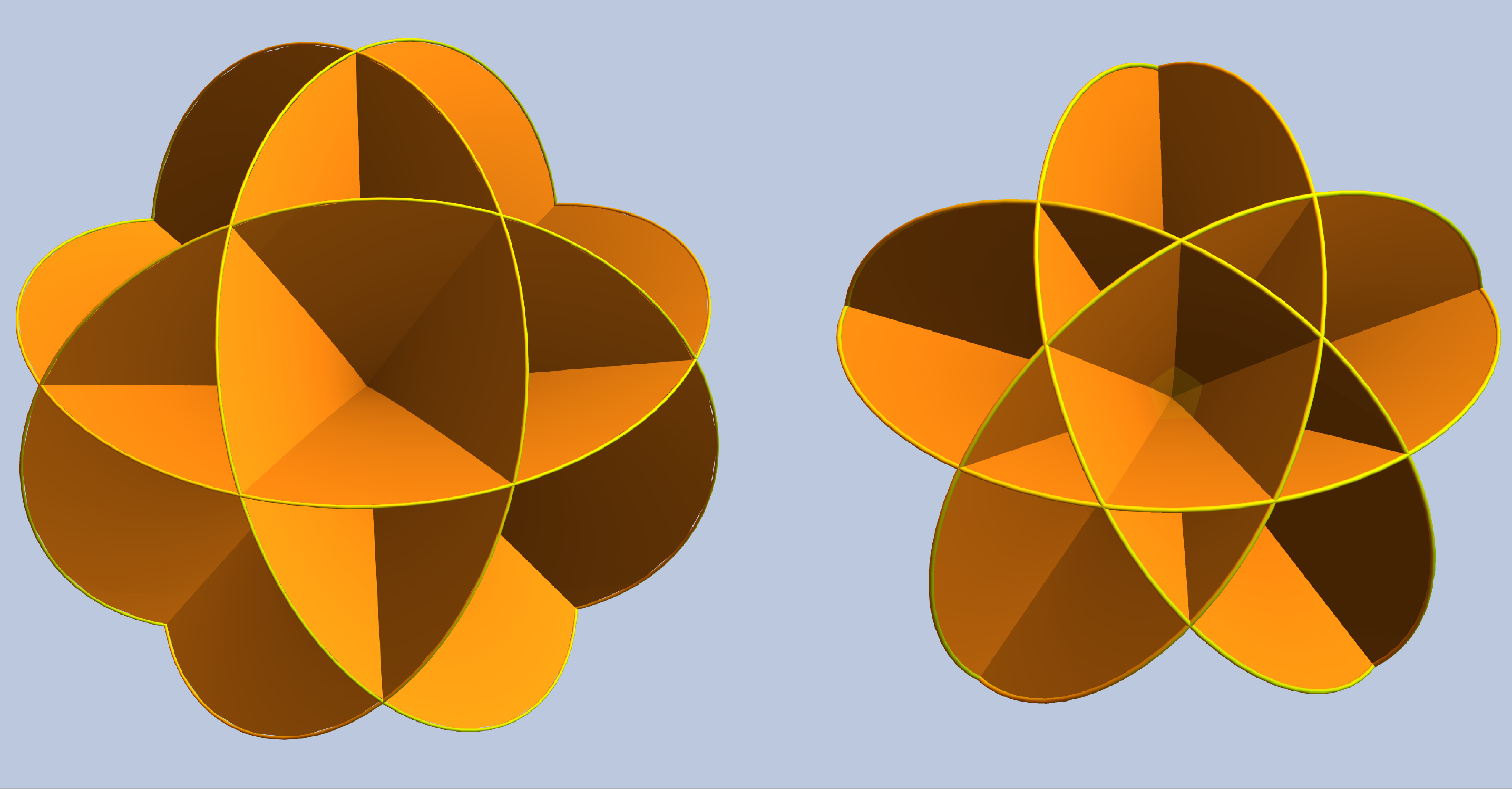}
\caption{\small{ The minimal surfaces tamed by $f_{5/3}$ (left) and by $f_{3/2}$ (right)}.}\label{FIGURA4}
\end{center}
\end{figure}

\begin{remark} 
If $m+n$ is even, the $W$-curve $f^\sharp_q$ is congruent to the Legendrian curve with Bryant's potentials $g=iz^{(n-m)/2}$ and 
$f=(m-n)z^n/n$. If $m+n$ is odd, $f^\sharp_q$ is congruent to the Legendrian curve with Bryant's potentials $g=-i\sqrt{m/n}z^{n-m}$ and 
$f=(m-n)m^{n/(m-n)}n^{m/(n-m)})z^{2n}$.
\end{remark}

% \vskip0.1cm
% \begin{figure}[h]
%\begin{center}
%\includegraphics[height=5.6cm,width=11.2cm]{profiles.pdf}
%\caption{\small{ The profiles of the Bryant's surfaces tamed by $f_{17/6}$ (left) and by $f_{16/7}$ (right)}}\label{FIGURA5}
%\end{center}
%\end{figure}

\begin{ex}\label{ex2}  
We now briefly describe the surfaces tamed by isotropic $W$-curves, with the exclusion of the cycles (i.e., $q=\pm\, 3,\pm \,1/3$), 
already considered in Example \ref{ex1}.  Let $q=m/n\in \Q\setminus \{\pm\,1, \pm\, 3, \pm\, 1/3\}$. 

(1) The minimal surface tamed by $f_q$ is a conformal Goursat transform of the branched minimal immersion with Weierstrass data  $(z^{h_1(q)},z^{h_2(q)}dz)$, where  $h_1(q)=\frac{1}{2}(n-m)(1+\varepsilon_{q})$ and $h_2(q)=\frac{1}{2}(m+n)(1+\varepsilon_{q})-1$. 
In particular, if $m=2\hat m+1$, $n=1$, then $h_1=h_2=\hat m$ and, if $m=2\hat m+1$ and $n=1$, we have $h_1=\hat m+1$ and 
$h_2=\hat m-1$. The minimal surfaces with this Weierstrass data are the  {\it Enneper surface of order $\hat m$} and the {\it $A_4$ surface of order $\hat m$}, respectively (cf. \cite[pp. 202-204]{DHS}).  
These are conformal Goursat transforms of each other.  The Enneper surface of order $\hat m$ possesses one non-embedded end of order $2\hat m+1$ at infinity. The $A_4$-surface of order $\hat m$ has one non-embedded end of order $2\hat m+1$ at infinity and one planar embedded end at the origin. Then, they are not classical Goursat transforms of each other. 
This shows that the (conformal) Goursat transform can modify the behaviour of the ends, even in the case of minimal surfaces 
tamed by rational isotropic curves. Figure \ref{FIGURA4} reproduces the minimal surfaces tamed by the isotropic $W$-curves 
$f_{5/3}$ and $f_{3/2}$, respectively.

%\vskip0.1cm
%\noindent 
(2) The CMC 1 surfaces of hyperbolic 3-space tamed by isotropic $W$-curves are Goursat transforms of the Catenoid cousin 
with rational parameter $\mu = (q-1)/2$ considered in \cite{Br1987}.  CMC 1 surfaces in hyperbolic space with and 
arbitrary number $\hat m$ of smooth ends and a rotational symmetry group of order $\hat m-1$ can be constructed 
from the CMC 1 surface tamed by $f_{2\hat m+1}$, via Goursat transformation \cite{Bo-Pe2005}. 
Figure \ref{FIGURA6} depicts a Goursat transform of the Catenoid cousin with parameter $5$, with seven ends. 
It is invariant by the group of order six generated by the rotation of angle $2\pi/6$ around the $z$-axis. 
This shows that the Catenoid cousin with rational parameter $(q-1)/2$ is a conformal Goursat transform of the minimal surfaces with Weierstrass data $h_1(q)$ and $h_2(q)$.

%\vskip0.1cm
%\noindent  
(3) The flat front of ${\mathcal H}^{3}$ tamed by $f_q$ is a Goursat transform of the flat fronts of revolution 
with rational parameter $\mu = (q-1)/(q+1)$ examined in \cite{KUY-O}. In particular, the flat front of revolution with 
rational parameter $\mu$ is a conformal Goursat transform of the Catenoid cousin with parameter $\mu/(\mu -1)$.
\end{ex}

 \begin{figure}[h]
\begin{center}
\includegraphics[height=4.5cm,width=13cm]{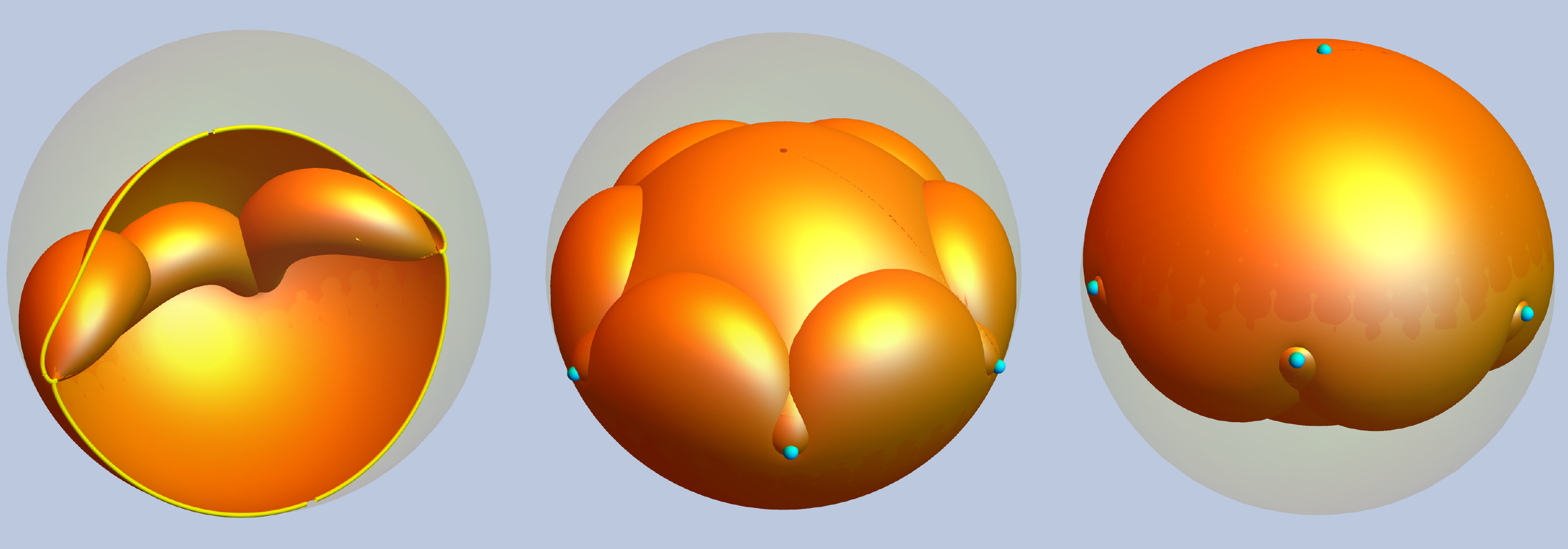}
\caption{\small{ A CMC 1 surface with seven smooth ends and symmetry group of order six: Goursat transform of the Bryant surface tamed by $f_{11}$. The points coloured in cyan are the ends.}}\label{FIGURA6}
\end{center}
\end{figure}

 \subsection{Isotropic curves with constant bending 
% and isotropic $W$-curves
 } 
 
 Let $f_{\kappa}$ be an isotropic curve with constant bending $\kappa$. 
 We say that $f$ is of {\it regular type} if  $\kappa\neq 1, -16/9$ and of {\it exceptional type}, otherwise. 
 Let ${\mathcal D}_4$ be the dihedral transformation group of $\dot{\C}$ generated by $z\to 1/z$ and $z\to -z$. 
 Let $|[z]|$ be the equivalence class of $z\in \dot{\C}$ in $\dot{\C}/{\mathcal D}_4$. 
 Given $\kappa\in \C\setminus \{ 1,-16/9\}$, we choose $c$ such that $c^2=\kappa$. 
 The equivalence class of  $\sqrt{(5c-4\sqrt{c^2-1})(5c+4\sqrt{c^2-1})^{-1}}$ does not depend on the choice 
 of $c$ and of the square roots. This originates a map
 $$
  {\bf r}:\C\setminus \{ 1,-16/9\}\ni k\mapsto r_k=|[\sqrt{(5c-4\sqrt{c^2-1})(5c+4\sqrt{c^2-1})^{-1}}]|\in \dot{\C}/{\mathcal D}_4.
      $$ 
   The range of ${\bf r}$ is $\dot{\C}/{\mathcal D}_4\setminus \{|[3]|\}$.  For instance, ${\bf r}_{\kappa_q}=|[q]|$, as can be easily seen from \eqref{bendingw}. Note that $\dot{\Q}$ is ${\mathcal D}_4$-stable. Its projection to  $\dot{\C}/{\mathcal D}_4$ is denoted by $|[\dot{\Q}]|$.
   
   The main result of the section is the following.
 
 \begin{thmx}\label{thm:cpt-cc}
 Let $f:S\to \Q_3$ be a compact isotropic curve with constant bending $\kappa$. 
 Then, $f$ is of general type, ${\bf r}(\kappa)$ belongs to $|[\dot{\Q}]|$, and $f$ is dominated by $f_q$, 
 where $q$ is the unique element of ${\bf r}(\kappa)$ strictly greater than $1$.
\end{thmx}

\noindent The proof of Theorem \ref{thm:cpt-cc} is based on the following three lemmas. 

 \begin{lemma}\label{lemma1}
 Let $f:S\to \Q_3$ be a compact isotropic curve of general type with constant bending $\kappa$.
% and of the general type. 
 If ${\bf r}_{\kappa} \in |[\Q]|$, then $f$ is dominated by the isotropic $W$-curve $f_q$, where $q$ is 
 the unique element of ${\bf r}_{\kappa}$ strictly greater than $1$.
\end{lemma}

\begin{proof}[Proof of Lemma \ref{lemma1}]
%As a preliminary point, let us recall two facts about $W$-curves :
We start by recalling the following two facts about isotropic $W$-curves.
\begin{itemize}

\item $f_q: \CP^1\to \Q_3$ is injective with at most two branch points at $0$ and $\infty$, with $f_q(0)=[{\bf e}_1\wedge {\bf e}_2]$ and 
$f_q(\infty)=[{\bf e}_3\wedge {\bf e}_4]$.

\item There exists a holomorphic frame field $A_q:\CP^1\setminus \{0,\infty\}\to \Sp(2,\C)$, such that
\begin{equation}\label{fr1}
A_q^{-1}dA_q =\frac{i}{\sqrt{10}}(s({\bf e}^1_2+{\bf e}^2_4+
{\bf e}^3_1-{\bf e}^4_3)-4(m^2+n^2)s^{-1}({\bf e}^4_2+\frac{3}{4}{\bf e}^2_1-\frac{3}{4}{\bf e}^3_4))\frac{dz}{z},
      \end{equation}
where $s$ is a fourth root of $-9m^4+82m^2n^2-9n^4$.
\end{itemize}
Next, we prove that there exists $X\in \Sp(2,\C)$ such that $f(S)\subseteq X\cdot f_q(\CP^1)$. Put
$c=-4i(m^2+n^2)s^{-2}$. Notice that $c^2=\kappa$.
Let $S_{\sharp}=S\setminus |\Delta_f|\cup|\Delta_{f^{\sharp}}|\cup |\Delta_{\delta_f}|$. Consider the universal covering
${\mathtt p}:\hat{S}_{\sharp}\to S_{\star}$. By construction, $\hat{S}_{\sharp}\not\cong\CP^1$.
%Otherwise, $S=S_{\sharp}\equiv \CP^1$. Then, $|\Delta_{\delta_f}|=\emptyset$ and $\delta_f$ 
%would be a non-zero holomorphic diferential. A contradiction. 
Hence, $\hat{S}_{\sharp}$ is contractible. From Lemma \ref{projreduction}, 
$f\circ {\mathtt p}$ admits a lift $A:\hat{S}_{\sharp}\to \Sp(2,\C)$ such that
\begin{equation}\label{fr2}
  \begin{cases}A^{-1}dA =\big(({\bf e}^1_2+{\bf e}^2_4+{\bf e}^3_1-{\bf e}^4_3)-c({\bf e}^4_2+\frac{3}{4}{\bf e}^2_1-
\frac{3}{4}{\bf e}^3_4)\big)\zeta,\quad 1<q<3,\\
A^{-1}dA =i\big(({\bf e}^1_2+{\bf e}^2_4+{\bf e}^3_1-{\bf e}^4_3)-c({\bf e}^4_2+\frac{3}{4}{\bf e}^2_1-
\frac{3}{4}{\bf e}^3_4)\big)\zeta,\quad  q>3,
    \end{cases}
    \end{equation}
where $\zeta$ is a nonzero holomorphic 1-form. Let $w:\hat{S}_{\sharp}\to \dot{\C}$ be defined by 
$$
     w={\rm exp}\left((-i)^{\frac{1-{\rm sign}(q-3)}{2}}10^{1/2}s^{-1}\int \zeta\right).
         $$
From \eqref{fr1} and \eqref{fr2}, we have $A^{-1}dA=w^*(A_q^{-1}dA_q)$. Then, there exists $X\in \Sp(2,\C)$, such that 
$A=X\cdot A_q\circ w$. 
This implies $f\circ {\mathtt p}= X\cdot f_q\circ w$. Hence, $f(S_{\sharp})\subseteq X\cdot f_q(\CP^1)$. 
By continuity, $f(S)\subseteq X\cdot f_q(\CP^1)$. 
Possibly replacing $f$ with $X^{-1}f$, we can assume $f(S)\subseteq f_q(\CP^1)$. Since $f_q$ is injective, there exists
 a unique map $h:S\to \CP^1$, such that $f=f_q\circ h$. 
 Let $S_{\star} =
 %X_{\star}= 
 S\setminus f^{-1}([{\bf e}_1\wedge {\bf e}_2])\cup f^{-1}([{\bf e}_3\wedge {\bf e}_4])$.
Taking into account that $f_q|_{\CP^1\setminus \{0,\infty\}}$ is immersive, then $h|_{S_{\star}}$ is holomorphic. 
Consider a point $p_0$ in $f^{-1}([{\bf e}_1\wedge {\bf e}_2])$. 
In an open neighborhood $U$ of $p_0$, we have 
$$
   f|_U=[({\bf e}_1+f_1^3{\bf e}_3+f_1^4{\bf e}_4)\wedge ({\bf e}_2+f^4_1{\bf e}_3+f_2^4{\bf e}_4)],
       $$ 
where $f^i_j$ are holomorphic on $U$ and $f^i_j (p_0)=0$. From \eqref{w}, we have
$$
   f^3_1=-\frac{m+n}{m-n}h^{(1+\varepsilon_q) m},\quad f^4_1=-\frac{2i\sqrt{mn}}{m-n}h^{\frac{(1+\varepsilon_q)(m+n)}{2}},\quad
     f^4_2=\frac{m+n}{m-n}h^{(1+\varepsilon_q) n}.
     $$
 Then $h$ is holomorphic on $U$. The same argument shows that $h$ is holomorphic also on a neighborhood of every point of 
 $f^{-1}([{\bf e}_3\wedge {\bf e}_4])$. This concludes the proof of Lemma \ref{lemma1}.
\end{proof}

 \begin{lemma}\label{lemma2}
% There are no compact isotropic curves with constant bending $\kappa$ and of generat type such that ${\bf r}_{\kappa}\notin |[\Q]|$.
There are no compact isotropic curve of general type with constant bending $\kappa$ such that ${\bf r}_{\kappa}\notin |[\Q]|$.
\end{lemma}

\begin{proof} 
%We premise few comments on the action of the maximal abelian subgroup 
%We start by making some comments on the action on $\Q_3$ of the maximal abelian subgroup
Consider the maximal abelian subgroup $\mathbb T^2 \subset \Sp(2,\C)$,
$$
   {\mathbb T}^2=\{u^{-1}{\bf e}^1_1+v^{-1}{\bf e}^2_2+u{\bf e}^3_3+v{\bf e}^4_4 \mid u,\,v\in \dot{\C}\}.
   $$ 
%   of $\Sp(2,\C)$.
%
Let $\hat {\mathcal U}\subset {\mathcal U}$ be the subset defined by 
%$$
% \hat {\mathcal U}=\{{\bf e}^1_1+{\bf e}^2_2+m^3_1{\bf e}_3^1+m^4_2{\bf e}^2_4+m^4_1({\bf e}^4_4+{\bf e}^2_3) \mid m^3_1,m^4_1,m^4_2\neq 0, \,m^3_1 m^4_2-(m^4_1)^2\neq 0\}.
%  $$
  $$ 
  %\mbox{ \small $
 \hat {\mathcal U}=\{[({\bf e}_1+ m^3_1{\bf e}_3 +m^4_1{\bf e}_4)
 \wedge ({\bf e}_2 +m^4_1{\bf e}_3+ m^4_2{\bf e}_4)] \mid m^3_1,m^4_1,m^4_2\neq 0, \,m^3_1 m^4_2 \neq (m^4_1)^2 \}. 
 %$}
  $$
Then, $\hat {\mathcal U}$ is ${\mathbb T}^2$-stable, the action of ${\mathbb T}^2$ 
%is free of 
has no fixed points
%fixed-point free 
and  
$$\Sigma=\{P_r \mid P_r=[({\bf e}_1+{\bf e}_3+{\bf e}_4)\wedge ({\bf e}_2+{\bf e}_3+r{\bf e}_4)],\, r\neq 0,1\}$$
is a slice. Let ${\mathcal O}_r$ be the ${\mathbb T}^2$-orbit through $P_r$. The orbits are the integral manifolds of 
the holomorphic completely integrable plane field distribution generated by the two fundamental vector fields of the action. 
In addition, ${\mathcal O}_r$  inherits a conformal structure from $\Q_3$. Let 
${\mathcal N}^+_r$ and ${\mathcal N}^+_r$ the two null distributions of ${\mathcal O}_r$. They can be viewed as completely 
integrable real distributions of rank 2.
\vskip0.1cm

\noindent Let $\kappa\in \C$, $\kappa\neq  1,-16/9$. Choose $c$ such that $c^2=\kappa$. 
Define $\lambda_1$ and $\lambda_2$ by $\lambda_1=\frac12(5c-4(c^2-1)^{1/2})^{1/2}$ and 
$\lambda_2=\frac12(5c+4(c^2-1)^{1/2})^{1/2}$. Observe that $\lambda_1\neq \lambda_2$ and  
$|[\lambda_1/\lambda_2]|={\mathbf r}_{\kappa}$. Let ${\bf v}_{\kappa},{\bf w}_{\kappa}:\C\to \C^4$ be defined by
\[
 \begin{split}{\bf v}_{\kappa} &= {\bf e}_1-\frac{\lambda_1+\lambda_2}{\lambda_1-\lambda_2}e^{2\lambda_1 z}{\bf e}_3-
\frac{2i\lambda_1^{1/2} \lambda_2^{1/2}}{\lambda_1-\lambda_2}e^{(\lambda_1+\lambda_2)z}{\bf e}_4,\\
{\bf w}_{\kappa}&={\bf e}_2-
\frac{2i\lambda_1^{1/2} \lambda_2^{1/2}}{\lambda_1-\lambda_2}e^{(\lambda_1+\lambda_2)z}{\bf e}_3+
\frac{\lambda_1+\lambda_2}{\lambda_1-\lambda_2}e^{2\lambda_2 z}{\bf e}_4.
\end{split}
   \]
Then, 
\begin{equation}\label{constben}\hat{f} _{\kappa}=[{\bf v}_{\kappa}\wedge {\bf w}_{\kappa}]:\C\to \hat{\mathcal U}\end{equation}
 is an isotropic curve  with constant bending $\kappa$. If $\lambda_1/\lambda_2$ is not rational, $\hat f_{\kappa}$ is injective and $\hat f_{\kappa}(\C)$ is contained in the orbit ${\mathcal O}_r$, $r=(\lambda_1+\lambda_2)^2/4\lambda_1\lambda_2$. Thus, $\hat f_{\kappa}$ is an integral manifold of one of the two null distributions of ${\mathcal O}_r$.
\vskip0.1cm

\noindent We are now in a position to conclude the proof of Lemma \ref{lemma2}. By contradiction,  let $f:S\to \Q_3$ be a compact isotropic curve with constant bending $\kappa$, with $\kappa \neq 1,-16/9$ and ${\bf r}_{\kappa}\notin |[\Q]|$. 
Consider $S_{\star}=S\setminus D$, $D=|\Delta_f|\cup |\Delta_{f^{\sharp}}|\cup |\Delta_{\delta}|$. 
Since $f$ and $\hat f_{\kappa}$ have the same constant bending,
proceeding as in the proof of Lemma \ref{lemma1}, 
there exists $X\in \Sp(2,\C)$, such that 
$f(S_{\star})\subseteq X\cdot f_{\kappa}(\C)$. Possibly replacing $f$ with $X^{-1}f$, we assume 
$f(S_{\star})\subseteq f_{\kappa}(\C)$.  If  ${\bf r}_{\kappa}\notin |[\Q]|$,
$f_{\kappa}$ is an integral manifold of a completely integrable distribution of ${\mathcal O}_r$. 
Thus (cf. \cite[Theorem 1.62, p. 47]{Wa}) there exists a differentiable map $h:S_{\star}\to \C$, such that $f|_{S_{\star}}=\hat f_{\kappa}\circ h$.  Since $f$ and $f_{\kappa}$ are holomorphic, also $h$ is holomorphic. 
Note that $f(S)\cap (\Q_3\setminus \mathcal{U})\neq \emptyset$. 
Otherwise, $f(S)\subset {\mathcal U}$ and ${\mathcal U}$ is biholomorphically equivalent to $\C^3$. 
Hence, by the maximum principle, $f$ would be constant. Next, pick $p_0\in S$ such that 
$f(p_0)\notin {\mathcal U}$.  The point $p_0$ belongs to $D$. In fact, if $p_0\notin D$, we would have 
$f(p_0)\in f(S_{\star})\subseteq f_{\kappa}(\C)\subset {\mathcal U}$. Choose a complex chart $(U,z)$ centered at $p_0$, such that 
$D\cap U=\{p_0\}$. The next argument is of a local nature, so we can think of $U$ as an open disk of the complex plane and 
take $p_0=0$. 
There exist a positive integer $\ell$ and holomorphic functions $m^3_1$,  $m^4_1$,  $m^3_2 :U\to \C$, not all vanishing 
at the origin and not identically $0$ on $U$, 
such that
%\footnote{A priori, the left and the right hand side can be different at $0$, in fact at $0$ the rank of the bivector 
%$(z^{\ell}{\bf e}_1+m^3_1{\bf e}_3+m^4_1{\bf e}_4)\wedge (z^{\ell}{\bf e}_2+m^4_1{\bf e}_3+m^4_2{\bf e}_4)$ can be $<2$.}
$$
   f|_{\dot{U}}=[({\bf e}_1+z^{-\ell}(m^3_1{\bf e}_3+m^4_1{\bf e}_4))\wedge ({\bf e}_2+z^{-\ell}(m^4_1{\bf e}_3+m^4_2{\bf e}_4))],
        $$
where  $\dot{U}=U\setminus \{p_0\}$. Therefore, on $\dot{U}$ we have
\begin{equation}\label{ffr1}\begin{cases}
z^{-\ell} m^3_1 = -\frac{\lambda_1+\lambda_2}{\lambda_1-\lambda_2}e^{2\lambda_1 h(z)},\\
 z^{-\ell} m^4_1 =-\frac{2i\lambda_1^{1/2} \lambda_2^{1/2}}{\lambda_1-\lambda_2}e^{(\lambda_1+\lambda_2)h(z)},\\
 z^{-\ell} m^4_2=\frac{\lambda_1+\lambda_2}{\lambda_1-\lambda_2}e^{2\lambda_2 h(z)}.
\end{cases}\end{equation}
Next, we use \eqref{ffr1} to get a contradiction. At least one of the functions $m^i_j$ is nonzero at the origin. 
Suppose $m^3_1(0)\neq 0$. Possibly taking a smaller neighborhood, we assume that $m^3_1$ is never zero. 
So we can write $m^3_1=e^{\alpha^3_1+i\beta^3_1}$ and $h=a+ib$, where $\alpha^3_1$, $\beta^3_1$, $a$, $b$ are real valued. 
Consider a small circle ${\rm C}_{\rho}=\{e^{\rho+it} \mid t\in \R\}$, such that the functions $m^4_1$ and $m^4_2$ 
are never zero on ${\rm C}_{\rho}$. Hence, we have
\begin{equation}\label{ffr2}
\begin{split}
m^3_1(e^{\rho+it}) & =e^{\alpha^3_1(t)+i\beta^3_1(t)},\quad m^4_1(e^{\rho+it})=e^{\alpha^4_1(t)+i\beta^4_1(t)},\\
m^4_2(e^{\rho+it}) & =e^{\alpha^4_2(t)+i\beta^4_2(t)},\quad   h(e^{\rho+it})=e^{\alpha(t)+i\beta(t)},
\end{split}
\end{equation}
where $\alpha,\alpha^3_1$, $\alpha^4_1$, $\alpha^4_1$, $\beta^3_1,\beta$ are periodic function of period $2\pi$ and 
$$ 
     \beta^4_1(t+2\pi)=\beta^4_1+2\pi n_1,\quad \beta^4_2(t+2\pi)=\beta^4_1+2\pi n_2,\quad n_1,n_2\in \Z.
          $$
From \eqref{ffr1} and \eqref{ffr2}, we have
\begin{equation}\label{ffr3}\begin{split}
2\lambda_1(\alpha(t)+i\beta(t))=\ell\rho +\alpha^3_1(t)+i(\ell t+\beta^3_1(t) +2\pi \hat n_1-\pi) -c_1,\\
  2\lambda_2(\alpha(t)+i\beta(t))=\ell\rho +\alpha^4_2(t)+i(\ell t+\beta^4_2(t) +2\pi \hat n_2-\pi) -c_2,
  \end{split}\end{equation}
where $\hat n_1$, $\hat n_2$ are two integers and $c_1,c_2$ are the constants defined by 
$e^{c_1}=(\lambda_1+\lambda_2)(\lambda_1-\lambda_2)^{-1}$ and 
$e^{c_2}=\sqrt{\lambda_1}\sqrt{\lambda_2}(\lambda_1-\lambda_2)^{-1}$. 
Since $|[\lambda_2/\lambda_1]|\in {\bf r}_{\kappa}\notin |[\Q]|$, we have $\lambda_2=r\lambda_1$, $r\notin \Q$.
From \eqref{ffr3}, we get
\begin{equation}\label{ffr4}
\begin{split}
\ell\rho +\alpha^4_2(t)+i(\ell t+\beta^4_2(t) &+2\pi \hat n_2-\pi) -c_2
 =\\& r(\ell\rho +\alpha^3_1(t)+i(\ell t+\beta^3_1(t) +2\pi \hat n_1-\pi) -c_1).
\end{split}
\end{equation}
Taking into account the periodicity and the quasi-periodicity of the functions $\alpha^i_j$ and $\beta^i_j$ and using \eqref{ffr4},
we conclude that $2\pi i(\ell + n_2)=2\pi i \ell r$, which is a contradiction.
With the same argument, one can show that assuming $m^4_2(0)\neq 0$, or $m^4_1(0)\neq 0$, 
would yield to a contradiction as well.
\end{proof}

 \begin{lemma}\label{lemma3}
 %There are no compact isotropic curves with constant bending and of exceptional type.
 There are no compact isotropic curve of exceptional type with constant bending.
\end{lemma}

\begin{proof} 
%We prove the Lemma in the case 
We shall give the proof in the case $\kappa=1$. The same argument can be used in the other case.  
The general structure of the reasoning is quite similar to that of the previous proof. 
First, observe that
\begin{equation}\label{fffr1}
  \hat f :  \C \ni z\mapsto [({\bf e}_1+e^{z}{\bf e}_3+z{\bf e}_4)\wedge ({\bf e}_2+z{\bf e}_3-e^{-z}{\bf e}_4)]\in \Q_3
    \end{equation}
is an embedded isotropic curve with constant bending $\kappa=1$. 
By contradiction, let $f:S\to \Q_3$ be a compact isotropic curve with bending 1.  Consider the exceptional locus $D=|\Delta_f|\cup |\Delta_{f^{\sharp}}|\cup |\Delta_{\delta}|$ and put $S_{\star}=S\setminus D$. 
In analogy with Lemma \ref{lemma1}, one  can prove that
$f(S_{\star})\subseteq X\cdot \hat f(\C)$, for some $X\in \Sp(2,\C)$. 
Replacing $f$ with $X^{-1}f$, we may assume $f(S_{\star})\subseteq \hat f(\C)$. 
Since $\hat f$ is an embedding, there exists a holomorphic map $h:S_{\star}\to \C$ such that 
$\hat f\circ h=f|_{S_{\star}}$. Choose $p_0\in S$, such that
$f(p_0)\notin {\mathcal U}$. As in Lemma \ref{lemma2}, $p_0\in D$. Consider a coordinate system $(U,z)$ centered at $p_0$,
such that 
$U\cap D=\{p_0\}$. The next argument is of a local nature, so we assume that $U$ is an open disk on the complex plane 
and take $p_0=0$. 
There exist  a positive integer $\ell$ and holomorphic functions $m^3_1,  m^4_1,  m^3_2 :U\to \C$, not all vanishing at the origin
and not identically zero on $U$, 
such that
\begin{equation}\label{fffr2}
  f|_{\dot{U}}=[({\bf e}_1+z^{-\ell}(m^3_1{\bf e}_3+m^4_1{\bf e}_4))\wedge ({\bf e}_2+z^{-\ell}(m^4_1{\bf e}_3+m^4_2{\bf e}_4))],
    \end{equation}
where $\dot{U}=U\setminus \{p_0\}$. Possibly shrinking $U$, there exist integers 
$0\le k_1^3,k_1^4,k_2^4\le \ell$, at least one of them zero,  and holomorphic functions 
$\mu^3_1,\mu^4_1,\mu^4_2$, such that 
$m^3_1=z^{k^3_1}e^{\mu^3_1}$, $m^4_1=z^{k^4_1}e^{\mu^4_1}$, $m^4_2=z^{k^4_2}e^{\mu^4_2}$.  Comparing 
\eqref{fffr1} and \eqref{fffr2}, we have
$$
    h=z^{k^4_1-\ell}e^{\mu^4_1},\quad e^{h}=z^{k^3_1-\ell}e^{\mu^3_1},\quad e^{-h}=-z^{k^4_2-\ell}e^{\mu^4_2}.
    %\quad {\rm on}
         $$
These equalities hold on  $\dot{U}$. We claim that $k^3_1-\ell=0$. Consider a circle $\{e^{r+it} \mid t\in \R\}$ contained in $U$. 
Then, 
$h(e^{r+it})= a(t)+ib(t)$ and 
$\mu^3_1(e^{r+it})=a^3_1(t)+ib^3_1(t)$, where $a$, $b$, $a^3_1$, $b^3_1$ are periodic, of period $2\pi$. The equality $e^{h}=z^{k^3_1-\ell}e^{\mu^3_1}$ gives 
$a(t)+ib(t)=(k^{3}_1-\ell)(r+it)+a^3_1(t)+ib^3_1(t)+2\pi \hat ni$, for some $\hat n\in \Z$. The periodicity of $a$, $b$, $a^3_1$ and $b^3_1$ implies $k^3_1-\ell=0$. 
The same argument shows that $k^4_2-\ell=0$. Hence $k^4_1=0$ and 
$h=z^{-\ell}e^{\mu^4_1}$, $e^{h}=e^{\mu^3_1}$, $e^{-h}=-e^{\mu^4_2}$ on $\dot{U}$. Thus, 
$e^{\mu^3_1}=e^{z^{-\ell}e^{\mu^4_1}}$. In turn, this implies
$\mu^3_1=z^{-\ell}e^{\mu^4_1}+2\pi \check{n}i$, $\check{n}\in \Z$. 
This gives the seeked contradiction,
%This conclusion is absurd 
since the left hand side is bounded near $0$ and the right hand side is unbounded.
\end{proof}

\begin{ex}\label{ex3} 
It is easy to show that the Catenoid cousins of ${\mathcal H}^3$ (cf. \cite{Br1987}) with parameters 
$\mu$ and $\mu(1-2\mu)^{-1}$ are Goursat transforms of each other. So, from our perspective, it suffices to examine Catenoid cousins with $\mu \in (-1/2,0)$. 
We exclude the case $\mu = -1/3$ corresponding to a pseudo-Cyclide. The Catenoid cousin with parameter $\mu \in (-1/2,0)$, 
$\mu \neq -1/3$, is tamed by the standard isotropic curve \eqref{constben} with constant bending 
$$
    \kappa_{\mu}=\frac{4(1+2\mu+2\mu^2)^2}{4+16\mu-7\mu^2-18\mu ^3-9\mu ^4}\in (-\infty,-16/9)\cup (1,+\infty).
        $$
Two rotationally invariant flat fonts (cf. \cite{KUY-O}) of ${\mathcal H}^3$ with parameters $m$ and $1/m$ are Goursat transforms of 
each other. So, we may assume $m\in (0,1)$. If $m=1/2$, the flat front is a pseudo-Cyclide. Excluding this case, 
a rotationally invariant flat front with parameter $m\in (0,1)$, $m\neq 1/2$, is tamed by the standard isotropic curve with constant bending
$$
    \kappa_m=\frac{4(1+m)^2}{4m^4-17m^2+4}.
         $$
In particular, the Catenoid cousin with parameter $\mu \in (-1/2,0)$ is a Goursat transform of the rotationally invariant flat 
front with parameter $m=-\mu/(1+\mu)\in (0,1)$. The Catenoid cousins or the rotationally invariant flat fronts tamed 
by $W$-curves are those with rational parameters.
\end{ex}

\begin{ex}\label{ex4} 
Consider the isotropic curve
$$\hat f(z)= [({\bf e}_1+e^{z}{\bf e}_3+z{\bf e}_4)\wedge ({\bf e}_2+z{\bf e}_3-e^{-z}{\bf e}_4)]$$ 
with conformal bending $\kappa=1$.
The minimal surface tamed by $\hat f$ is the Catenoid. 
%in its standarad parameterization $(u,v)\to (\cos(v) \cosh(u), -\sin(v )\cosh(u), u)$.  
Consider
$$
  \hat f_{b}=\big(\cosh(\frac{b}{2}){I}_{4}+\sinh(\frac{b}{2})({\bf e}^1_2+{\bf e}^2_1-{\bf e}^3_4-{\bf e}^4_3)\big)\hat f,\quad b\in \R.
    $$
The 1-parameter family $\{\hat{\phi}_b\}_{b\in \R}$ of minimal surfaces tamed by $\{\hat f_b\}_{b\in \R}$ is  the  {\it Bonnet deformation of the Catenoid}  \cite{Bonnet, JMN, MN-Abd}.  With the exception of the Enneper surface and planes, the Bonnet family exhaust (up to similarity transformations of $\R^3$) the class of minimal surfaces with plane line of curvature. Note that all Bonnet minimal surfaces are conformal Goursat transforms of the catenoid. The minimal surface $\check \phi_b$  tamed by 
$\check  f_{b}=\big((-1)^{1/4}({\bf e}^1_1+ {\bf e}^2_2)- (-1)^{-1/4}({\bf e}_3^3+{\bf e}^4_4) \big)\hat f_b$ is the associate of $\hat \phi_b$ 
(for $b=0$, we have the Helicoid). Then,  $\{\check {\phi}_b\}_{b\in \R}$ is the {\it Thomsen deformation} of the Helicoid  
\cite{Thomsen}. The Thomsen surfaces, the Enneper surface and the plane are (up to a similarity of $\R^3$) the only minimal surfaces  
which are affine minimal as well \cite{Blaschke1923}. 
By construction, the Thomsen surfaces are conformal Goursat transforms of the Catenoid.

%parametric equations of minimal surface $\hat \phi_b$ tamed by $\hat f_b$ are
%$$\begin{cases}
%x_{b}(u,v)&=\cos(v)\cosh(u)\\
%y_{b}(u,v)&=-\cosh(b)\sin(v)\cosh(u)-\sinh(b)v,\\
%z_{b}(u,v)&=-\sinh(b)\cos(v)\sinh(u)+\cosh(b)u,
%\end{cases}
%$$
%The $1$-parameter family $\{\hat{\phi}_b\}_{b\in \R}$   is the  {\it Bonnet deformation of the Catenoid}  \cite{Bonnet,MN-Abd,Cho-Ogata}.  With the exception of the Enneper surfaces, the Bonnet family exhaust (up to similarity transformations of $\R^3$) the class of minimal surfaces with plane line of curvature. This shows that the Bonnet's minimal surfaces are Goursat transforms of the Catenoid.
\end{ex}

\begin{ex}\label{ex5}  
Up to a linear change of the independent variable, the CMC 1 surfaces of 
${\mathcal H}^3$ and ${\mathcal H}^{2,1}$  tamed by the isotropic curve with constant bending $-16/9$ are the 
Enneper cousin and the spacelike Enneper cousin considered in \cite{Br1987, Lee2005}. 
They are conformal Goursat transforms of each other.
\end{ex}

\end{document}